\documentclass[11pt]{article}

\usepackage{amssymb,amsmath}
\usepackage{amsthm}
\usepackage{hyperref}
\usepackage{mathrsfs}
\usepackage{textcomp}
\hypersetup{
    colorlinks,%
    citecolor=black,%
    filecolor=black,%
    linkcolor=black,%
    urlcolor=black
}

\usepackage{verbatim}

\makeatletter

\newdimen\bibspace
\renewenvironment{thebibliography}[1]{%
 \section*{\refname 
       \@mkboth{\MakeUppercase\refname}{\MakeUppercase\refname}}%
     \list{\@biblabel{\@arabic\c@enumiv}}%
          {\settowidth\labelwidth{\@biblabel{#1}}%
           \leftmargin\labelwidth
           \advance\leftmargin\labelsep
           \itemsep\bibspace
           \parsep\z@skip     %
           \@openbib@code
           \usecounter{enumiv}%
           \let\p@enumiv\@empty
           \renewcommand\theenumiv{\@arabic\c@enumiv}}%
     \sloppy\clubpenalty4000\widowpenalty4000%
     \sfcode`\.\@m}
    {\def\@noitemerr
      {\@latex@warning{Empty `thebibliography' environment}}%
     \endlist}

\makeatother

\makeatletter

\newtheorem{thm}{Theorem}[section]
\newtheorem{lem}[thm]{Lemma}

\newtheorem{cor}[thm]{Corollary}
\newtheorem{rem}[thm]{Remark}

\numberwithin{equation}{section}

\def\XXint#1#2#3{{\setbox0=\hbox{$#1{#2#3}{\int}$}
  \vcenter{\hbox{$#2#3$}}\kern-.5\wd0}}

\newcommand{\al}{\alpha}                \newcommand{\lda}{\lambda}
                \newcommand{\pa}{\partial}
\newcommand{\va}{\varepsilon}           \newcommand{\ud}{\mathrm{d}}
\newcommand{\be}{\begin{equation}}      \newcommand{\ee}{\end{equation}}
                 
\newcommand{\Lda}{\Lambda}              
\newcommand{\R}{\mathbb{R}}

\begin{document}

\title{\textbf{Schauder estimates for  nonlocal  fully nonlinear equations}
\bigskip}

\author{Tianling Jin\quad and \quad Jingang Xiong}

\date{\today}

\maketitle

\begin{abstract}
In this paper, we establish pointwise  Schauder estimates for solutions of nonlocal  fully nonlinear elliptic equations by perturbative arguments.
A key ingredient is a  recursive Evans-Krylov theorem for nonlocal  fully nonlinear translation invariant equations.
\end{abstract}

\section{Introduction}
Integro-differential equations, which are usually called nonlocal equations nowadays, appear naturally when studying discontinuous stochastic process.  
In a series papers of L. Caffarelli and L. Silvestre \cite{CS09, CS10, CS11}, regularities of solutions of nonlocal fully nonlinear elliptic equations such as H\"older estimates, $C^{1+\alpha}$ estimates, Cordes-Nirenberg type estimates and Evans-Krylov theorem were established.  In this paper, we shall prove Schauder estimates for nonlocal fully nonlinear elliptic equations of the type:
\be\label{eq:main schauder}
\inf_{a\in\mathcal A} \left\{\int_{\R^n} \delta u(x,y)K_a(x,y)\,\ud y\right\}=f(x)\quad\mbox{in }B_5,
\ee
where $\delta u(x,y)=u(x+y)+u(x-y)-2u(x)$, $\mathcal A$ is an index set, and each $K_a$ is a positive kernel.
We will restrict our attention to symmetric kernels which satisfy
\be\label{eq:symmetry}
K(x,y)=K(x,-y).
\ee
We also assume that the kernels are uniformly elliptic
\be\label{eq:elliptic}
\frac{(2-\sigma)\lda }{|y|^{n+\sigma}} \leq K(x,y)\leq \frac{(2-\sigma)\Lda }{|y|^{n+\sigma}}
\ee
for some $0<\lda\le\Lda<\infty$, which is an essential assumption leading to local regularizations. Finally,  we suppose that the kernels are $C^2$ away from the origin and satisfy
\be\label{eq:kernel smooth1}
|\nabla^i_y K(x,y)|\leq \frac{\Lda}{|y|^{n+\sigma+i}},\quad i=1,2.
\ee
We say that a kernel $K\in \mathscr{L}_0(\lda, \Lda, \sigma)$ if $K$ satisfies \eqref{eq:symmetry} and \eqref{eq:elliptic}, and $K\in \mathscr{L}_2(\lda, \Lda, \sigma)$ if $K$ satisfies \eqref{eq:symmetry}, \eqref{eq:elliptic} and \eqref{eq:kernel smooth1}. In this paper, all the solutions of nonlocal equations are understood in the viscosity sense, where the definitions of such solutions can be found in \cite{CS09}.  

One way to obtain Schauder estimates is that first we prove high regularity for solutions of translation invariant (or ``constant coefficients") equations, and then use perturbative arguments or approximations. In our case, the regularities for translation invariant equations should be the Evans-Krylov theorem for nonlocal fully nonlinear equations proved in \cite{CS11}, which states that: If $u$ is a bounded solution of
\[
\inf_{a\in\mathcal A} \left\{\int_{\R^n} \delta u(x,y)K_a(y)\,\ud y\right\}=0\quad\mbox{in }B_5,
\]
where every $K_a(y)\in\mathscr L_2 (\lda,\Lda, \sigma)$ with $\sigma\ge\sigma_0>0$. 
Then, $u\in C^{\sigma+\bar\al}(B_1)$ for some $\bar\al>0$. Moreover,
\be\label{eq:E-K-1}
\|u\|_{C^{\sigma+\bar\al}(B_{1})}\le N_{ek}\|u\|_{L^\infty(\R^n)},
\ee
where both $\bar\al$ and $N_{ek}$ are positive constants depending only on $n,\sigma_0,\lda, \Lda$. Note that $\bar\al$ and $N_{ek}$ do not depend on $\sigma$, and thus, do not blow up as $\sigma\to 2$. The result becomes most interesting when $\sigma$ is close to $2$ and $\sigma+\bar\al>2$. If we let $\sigma\to 2$, then it recovers the theorem of Evans and Krylov about the regularity of solutions to concave uniformly elliptic PDEs of second order. 

Throughout the paper, we will always denote $\bar\al$ as the one in \eqref{eq:E-K-1} without otherwise stated.

In the step of approximations to obtain Schauder estimates at $x=0$, it usually requires that the coefficients of the equations, which in our case are $K(x,y)$ and $f(x)$, are H\"older continuous at $x=0$ in some sense. For the right-hand side $f(x)$, we assume $f$ satisfies the standard H\"older condition that
\be\label{eq:holder f}
|f(x)-f(0)|\le M_f |x|^\al \quad \mbox{and}\quad|f(x)|\le M_f
\ee
for all $x\in B_5$, where $M_f$ is a nonnegative constant. 

For the kernel $K$, one may impose different types of H\"older conditions. Here, we focus on the (most delicate, as explained below) case that $\sigma+\bar\al-2\ge\gamma_0>0$, and we will assume the kernels satisfy
\be\label{eq:holder K}
\int_{\R^n}|K(x,y)-K(0,y)|\min(|y|^2,r^2)dy \le \Lambda |x|^{\alpha} r^{2-\sigma}
\ee
for all $r\in (0,1]$, $x\in B_5$.

For $s\in\mathbb R$, $[s]$ denotes the largest integer that is less than or equals to $s$. Our main result  is the following pointwise Schauder estimates for solutions of \eqref{eq:main schauder}. Recall that $\bar\alpha$ is the one in \eqref{eq:E-K-1}.
\begin{thm}\label{thm:schauder-nl}
Assume every $K_a(x,y)\in\mathscr L_2 (\lda,\Lda, \sigma)$ satisfies \eqref{eq:holder K} with $\al\in(0,\bar\al)$, $\sigma+\bar\al-2\ge\gamma_0>0$ and $|\sigma+\al- 2|\ge \va_0>0$. Suppose that $f$ satisfies \eqref{eq:holder f}. If $u$ is a bounded viscosity solution of \eqref{eq:main schauder}, then there exists a polynomial $P(x)$ of degree $[\sigma+\al]$ such that for $x\in B_1$,
\be\label{eq:schauder estimate}
\begin{split}
|u(x)-P(x)|&\leq C\left(\|u\|_{L^\infty(\R^n)}+ M_f\right)|x|^{\sigma+\al};\\
|\nabla^j P(0)|&\leq C\left(\|u\|_{L^\infty(\R^n)}+ M_f\right),\ j=0,\cdots, [\sigma+\al],
\end{split}
\ee
where $C$ is a positive constant depending only on $\lda,\Lda, n, \bar\al, \al, \va_0$ and $\gamma_0$.
\end{thm}
Roughly speaking, Theorem \ref{thm:schauder-nl} states that if $K$ and $f$ are of $C^\al$ at $x=0$ in the sense of \eqref{eq:holder K} and \eqref{eq:holder f}, respectively, then the solution $u$ of \eqref{eq:main schauder} is precisely of $C^{\sigma+\al}$ at $x=0$. Moreover, the constant $C$ in \eqref{eq:schauder estimate} does not depend on $\sigma$, and hence, does not blow up as $\sigma\to 2$.  

Various Schauder estimates for solutions of some  nonlocal linear equations were obtained before by R.F. Bass \cite{BR}, R. Mikulevicius and H. Pragarauskas \cite{MP}, H. Dong and D. Kim \cite{DK}, B. Barrera, A. Figalli and E. Valdinoci \cite{BFV}, D. Kriventsov \cite{KD},  as well as the authors \cite{JX}. The results in \cite{BFV} have applications to nonlocal minimal surfaces. The equations considered in \cite{BR, MP, DK, KD} are of rough kernels, i.e., without the assumption \eqref{eq:kernel smooth1}. Also in \cite{KD}, D. Kriventsov proved $C^{1+\alpha}$ estimates for nonlocal fully nonlinear equations with rough kernels when the order of the equation $s>1$ by perturbative arguments. Later, J. Serra \cite{JS1} extended this result in \cite{KD} to parabolic equations and used a different method.  In \cite{KRS},  M. Kassmann, M. Rang and R. W. Schwab proved H\"older regularity results for those nonlocal equations whose ellipticity bounds are strongly directionally dependent.  Recently, X. Ros-Oton and J. Serra \cite{ROS} studied boundary regularity for nonlocal fully nonlinear equations. One may see, e.g., \cite{BCI, DK1, GS} for more regularity results on nonlocal elliptic equations.

For the H\"older condition \eqref{eq:holder K} on the kernels, one can check that it will hold if the kernels satisfy the pointwise H\"older continuous condition $|K(x,y)-K(0,y)|\le \Lda (2-\sigma)|x|^\al |y|^{-n-\sigma}$. In the case of $\sigma+\bar\al<2$, all of our arguments still work except that one needs to  change  the condition \eqref{eq:holder K} to \eqref{eq:holder K-2} or \eqref{eq:holder K-3}, since the approximation solutions will be of only $C^{\sigma+\bar\al}$; see Remark \ref{rem:less than 2} and Corollary \ref{cor:unified}.

In the case of second order partial differential equations $F(\nabla^2 u,x)=0$, to show that $u\in C^{2+\al}$, we usually use second order polynomials $p(x)$ to approximate $u$ (see \cite{Ca,CC}), in which one implicit convenience is that $\nabla^2 p(x)$ is a constant function. In the nonlocal case, to prove $C^{\sigma+\al}$ estimates of solutions to \eqref{eq:main schauder} for $\sigma+\al>2$, second order polynomial approximation does not seem to work directly, since first of all, for a second order polynomial $p(x)$, it grows too fast at infinity so that $\delta p(x,y)K(y)$ is not integrable; and secondly, in general $\int_{\R^n}\delta \tilde p(x,y)K(y)\ud y$ will not be a constant function for any cut-off $\tilde p(x)$ of $p(x)$ so that we cannot apply Evans-Krylov theorem during the approximation and will lose control of the error. Another common difficulty in approximation arguments to obtain regularities for nonlocal equations is to control the error outside of the balls in the iteration,  which may result in a slight loss of regularity as in \cite{CS10} compared to second order equations.  Instead of polynomials, we will approximate the genuine solution by solutions of ``constant coefficients" equations, which is inspired by \cite{Ca,LN}. In this way, we do not need to worry about either polynomials or the errors coming from the infinity. But a new difficulty arises for fully nonlinear equations (which does not appear in the case of linear equations): the Evans-Krylov theorem in \cite{CS11} cannot be applied to obtain the uniform estimates for the sequence of approximation solutions to those ``constant coefficients'' equations; see also Remark \ref{rem:how to appr}. This leads us to establish a recursive Evans-Krylov theorem in Theorem \ref{thm:E-K-3} to overcome this difficulty.

Our paper is organized as follows. In Section \ref{sec:EK}, we prove Theorem \ref{thm:E-K-3}, a recursive Evans-Krylov theorem for nonlocal fully nonlinear equations, where we adapt the proofs in \cite{CS11} with delicate decomposition and cut-offs arguments. In Section \ref{sec:schauder}, we will use Theorem \ref{thm:E-K-3} and perturbative arguments to prove the Schauder estimates in Theorem \ref{thm:schauder-nl}. In the Appendix, we recall some definitions and notions of nonlocal operators from \cite{CS10}, and establish two approximation lemmas for our own purposes, which are variants of that in \cite{CS10}.

After we finished our paper, we learned from Joaquim Serra that he has a preprint \cite{JS} on estimates for concave nonlocal fully nonlinear elliptic equations with rough kernels, where Schauder estimates are obtained using very different methods. 

\bigskip

\noindent\textbf{Acknowledgements:} We would like to thank Professor Luis Silvestre for many useful discussions. We also thank Professor YanYan Li for his interests and constant encouragement. Tianling Jin was supported in part by NSF grant DMS-1362525. Jingang Xiong was supported in part by the First Class Postdoctoral Science Foundation of China (No. 2012M520002) and Beijing Municipal Commission of Education for the Supervisor of Excellent Doctoral Dissertation (20131002701).

\section{A recursive Evans-Krylov theorem}\label{sec:EK}

\subsection{Statements and ideas of the proof}\label{sec:statement}

If we re-examine the proof of the nonlocal Evans-Krylov theorem in \cite{CS11}, we can show the following theorem with few modification.
\begin{thm}\label{thm:E-K-2}
Assume that every $K_a(y)\in\mathscr L_2 (\lda,\Lda, \sigma)$ with $2>\sigma\ge\sigma_0>1$ and every $b_a$ is a constant. If $w$ is a bounded solution of
\[
\inf_{a\in\mathcal A} \left\{\int_{\R^n} \delta w(x,y)K_a(y)\,\ud y+b_a\right\}=0\quad\mbox{in }B_5,
\]
then, $w\in C^{\sigma+\bar\al}(B_1)$, and there holds
\[
\|w\|_{C^{\sigma+\bar\al}(B_{1})}\le N_{ek}(\|w\|_{L^\infty(\R^n)}+|\inf_{a} b_a|),
\]
where both $\bar\al$ and $N_{ek}$ are the same as those in \eqref{eq:E-K-1}.
\end{thm}

The recursive Evans-Krylov theorem we are going to show is the following.

\begin{thm}\label{thm:E-K-3}
Assume that every $b_a$ is a constant, $K_a(y)\in\mathscr L_2 (\lda,\Lda, \sigma)$ with $2>\sigma\ge\sigma_0>1$. For each $m\in \mathbb{N}\cup\{0\}$,  let $\{v_\ell\}_{\ell=0}^m$ be a sequence of functions satisfying \be\label{eq:main}
\inf_{a\in\mathcal A} \left\{\int_{\R^n} \sum_{\ell=0}^j \rho^{-(j-\ell)(\sigma+\al)} \delta v_\ell(\rho^{j-\ell}x,\rho^{j-\ell}y)K^{(j)}_a(y)\,\ud y+\rho^{-j\al}b_a\right\}=0\quad\mbox{in }B_5
\ee
in viscosity sense for all $0\le j\le m$, where $K^{(j)}_a(x)=\rho^{j(n+\sigma)}K_a(\rho^j x)$, $\rho\in (0,1)$, $\al\in (0,\bar \al)$. Suppose that$\|v_\ell\|_{L^\infty(\R^n)}\le 1$ for all $\ell$ and $|\inf_{a\in \mathcal A}b_a|\le 1$. Then, $v_\ell\in C^{\sigma+\bar\al}(B_1)$, and there exist constants $C>0$ and $\rho_0\in (0,1/100)$, both of which depend only on $n,\sigma_0,\lda $, $ \Lda$, $\bar \al$ and $\al$, such that if $\rho\le\rho_0$ then we have
\be\label{eq:E-K-3}
\|v_\ell\|_{C^{\sigma+\bar\al}(B_{1})}\le C \quad \forall~ \ell= 0,1,\dots,m.
\ee
\end{thm}

The rest of this section will be devoted to proving Theorem \ref{thm:E-K-3}. The regularity of $v_{i+1}$ follows from the Evans-Krylov theorem in \cite{CS11}. But if one applies the estimate (1.4) in \cite{CS11} to $v_\ell$ directly, one will get their $C^{\sigma+\bar\al}$ estimates depending on $\ell$ and $\rho$.  Our goal is to prove the estimate \eqref{eq:E-K-3} which is independent of both $\ell$ and $\rho$. 

A constant $C$ is said to be a universal constant if $C$ only depends on $n,\sigma_0,\lda $, $ \Lda, \al$ and $\bar\al$. Throughout this section, all the constants denoted as $C$ will be universal constants, and it may vary from lines to lines.

Let $M>>1$  be a universal constant which will be fixed later. Replacing $v_\ell$ by $v_\ell/M$, we may assume that
\[
\|v_\ell\|_{L^\infty(\R^n)}\le 1/M \quad  \mbox{and} \quad |\inf_{a\in \mathcal A}b_a|\le 1/M.
\]
Then our goal is to show that 
\[
\|v_\ell\|_{C^{\sigma+\bar\al}(B_{1})}\le 1 \quad \forall~ \ell=0,1,\cdots,m.
\]
The proof is by induction on $m$. When $m=0$, then by Theorem \ref{thm:E-K-2}, \eqref{eq:E-K-3} holds for $M=2N_{ek}$. We assume that Theorem \ref{thm:E-K-3} holds up to $m= i$ for some $i\ge 0$, and we are going to show that it holds for $i+1$ as well.

It follows from the induction hypothesis and the $i+1$ equations for $v_0,\dots,v_i$ that
\[
\|v_\ell\|_{C^{\sigma+\bar\al}(B_{1})}\le 1,\quad \forall~\ell=0,1,\dots,i.
\]
We are going to show
\be \label{eq:ek v i+1}
\|v_{i+1}\|_{C^{\sigma+\bar\al}(B_{1})}\le 1.
\ee

To illustrate the idea of our proof, let us first consider the second order fully nonlinear elliptic equations
\be\label{eq:2nd ek}
F(D^2 u):=\inf_{k\in\mathcal K} a_{ij}^{(k)}u_{ij}=0\quad\mbox{in }B_5,
\ee
where $\mathcal K$ is an index set, and $\lda I\le (a_{ij}^{(k)})\le \Lda I$ for all $k\in\mathcal K$. By the Evans-Krylov theorem, for every viscosity solution $u$ of \eqref{eq:2nd ek}, we have
\[
\|u\|_{C^{2+\bar\al}(B_1)}\le N_{ek}\|u\|_{L^\infty(B_5)}.
\]
Suppose that there exists a sequence of functions $\{v_\ell\}_{\ell=0}^m$ satisfying
\[
F\Big(\sum_{\ell=0}^j D^2\big(\rho^{-(j-\ell)(2+\al)}v_\ell(\rho^{j-\ell} x)\big)\Big)=0\quad\mbox{in }B_5
\]
in viscosity sense for all $0\le j\le m$, and $\|v_\ell\|_{L^\infty(B_5)}\le 1/M$ for all $\ell$. Suppose that up to $m=i$ for some $i\ge 0$ there holds
\[
\|v_\ell\|_{C^{2+\bar\al}(B_1)}\le 1\quad\mbox{for all }\ell=0,1,\dots,m.
\]
We are going to show this holds for $m=i+1$ as well.  For $\ell=0,\dots, i$, we let $P_\ell$ be the second order Taylor expansion polynomial of $v_\ell$ at $x=0$. Let
\[
\tilde v_{i+1}=v_{i+1}+\sum_{\ell=0}^{i} \rho^{-(i+1-\ell)(2+\al)}(v_\ell-P_\ell)(\rho^{i+1-\ell} x).
\]
Then
\[
G(D^2 \tilde v_{i+1}):=F\Big(D^2 \tilde v_{i+1}+\sum_{\ell=0}^i D^2\big(\rho^{-(i+1-\ell)(2+\al)}P_\ell(\rho^{i+1-\ell} x)\big)\Big)=0\quad\mbox{in }B_5.
\]
It is clear that $G(\cdot)$ is uniformly elliptic and concave. Since,
\be\label{eq:hessian constant}
\sum_{\ell=0}^i D^2\big(\rho^{-(i+1-\ell)(2+\al)}P_\ell(\rho^{i+1-\ell} x)\big) \mbox{ is a constant matrix,}
\ee
and 
\be\label{eq:F0}
F\Big(\sum_{\ell=0}^i D^2\big(\rho^{-(i+1-\ell)(2+\al)}P_\ell(\rho^{i+1-\ell} x)\big)\Big)=0,
\ee
we have $G(0)=0$. By the Evans-Krylov theorem, 
\[
\|\tilde v_{i+1}\|_{C^{2+\bar\al}(B_1)}\le N_{ek}\|\tilde v_{i+1}\|_{L^\infty(B_5)}.
\]
Since
\[
\begin{split}
&\|\sum_{\ell=0}^{i} \rho^{-(i+1-\ell)(2+\al)}(v_\ell-P_\ell)(\rho^{i+1-\ell} x)\|_{L^\infty(B_5)}\\
&\le 5^{2+\bar\al}\sum_{\ell=0}^i \rho^{(i+1-\ell)(\bar\al-\al)}\le \frac{5^3\rho^{\bar\al-\al}}{1-\rho^{\bar\al-\al}}
\end{split}
\]
and
\[
\begin{split}
&\|\sum_{\ell=0}^{i} \rho^{-(i+1-\ell)(2+\al)}(v_\ell-P_\ell)(\rho^{i+1-\ell} x)\|_{C^{2+\bar\al}(B_1)}\\
&\le 4\cdot 5^{2+\bar\al}\sum_{\ell=0}^i \rho^{(i+1-\ell)(\bar\al-\al)}\le \frac{5^4 \rho^{\bar\al-\al}}{1-\rho^{\bar\al-\al}},
\end{split}
\]
it follows that
\[
\|v_{i+1}\|_{C^{2+\bar\al}(B_1)}\le N_{ek}\Big( 1/M+ \frac{5^3}{1-\rho^{\bar\al-\al}}\rho^{\bar\al-\al}\Big)+ \frac{5^4}{1-\rho^{\bar\al-\al}}\rho^{\bar\al-\al}\le1
\]
if we choose $M$ sufficiently large and $\rho_0$ sufficiently small.

From this proof for the second order case, we see that the idea is to decompose $v_\ell$ as $(v_\ell-P_\ell)+P_\ell$, and apply Evans-Krylov theorem to the equation for $\tilde v_{i+1}$ which is $v_{i+1}$ plus those rescaled $(v_\ell-P_\ell)$.  In this step, we used \eqref{eq:hessian constant} and \eqref{eq:F0}.

In the nonlocal fully nonlinear case \eqref{eq:main}, we are going to use the same idea of decomposing $v_\ell$ and studying the equation of $\tilde v_{i+1}$. However, there is a difficulty that $\delta P_\ell(x,y)K(y)$ is not integrable and $\int_{\R^n}\delta \tilde P_\ell(x,y)K(y)\ud y$ will never be a constant for any cut-off $\tilde P_\ell$ of $P_\ell$. Thus, we are not be able to use the Evans-Krylov theorem proved in \cite{CS11}. Instead, we are going to employ the proofs in \cite{CS11} to prove the $C^{\sigma+\bar\al}$ estimate for $v_{i+1}$. A delicate part is that we need to decompose $v_\ell$ in an appropriate way. We start with some preliminaries in the following.

\subsection{Preliminaries}
For a kernel $K(y)$, we denote
\[
Lu(x)=\int_{\R^n} \delta u(x,y)K(y)\,\ud y,
\]
We will also say $L\in\mathscr L_2 (\lda,\Lda, \sigma)$ (or $\mathscr L_0 (\lda,\Lda, \sigma)$) if $K\in\mathscr L_2 (\lda,\Lda, \sigma)$ (or $\mathscr L_0 (\lda,\Lda, \sigma)$).
\begin{lem}\label{lem:derivative-1}
Suppose that $u\in C^{4}(B_2)\cap L^\infty(\R^n)$ and $K(y)\in\mathscr L_2 (\lda,\Lda, \sigma)$. Then
\[
\|Lu\|_{C^2(B_1)}\le C (\|u\|_{C^{4}(B_2)}+\|u\|_{L^\infty(\R^n)}),
\]
where $C$ is a positive constant depending only on $\al$, $\sigma_0$ and $\Lda$.
\end{lem}
\begin{proof}
Let $\eta\in C_c^\infty(B_{3/2})$ and $\eta\equiv 1$ in $B_{5/4}$. Then
\[
 Lu=L(\eta u)+ L((1-\eta)u).
\]
It is clear that $\pa_{ij}(L(\eta u))= L(\pa_{ij}(\eta u))$, from which it follows that
\[
\|L(\eta u)\|_{C^2(B_1)}\le C (\|u\|_{C^{4}(B_2)}+\|u\|_{L^\infty(\R^n)}).
\]
For the second term, we have $1-\eta(x)=0$ if $x\in B_1$, and thus
\[
\begin{split}
 L((1-\eta)u)(x)&=\int_{\R^n} (1-\eta(x+y))u(x+y)K(y)\ud y\\
 &=\int_{\R^n\setminus B_{5/4}} (1-\eta(y))u(y)K(y-x)\ud y.
 \end{split}
\]
The lemma follows immediately since $K(y)\in\mathscr L_2 (\lda,\Lda, \sigma)$.
\end{proof}

\begin{lem}\label{lem:derivative-2}
Suppose that $u\in C^{\sigma+\al}(\R^n)$, $0\le K(y)\le (2-\sigma)\Lda |y|^{-n-\sigma}$ and $K(y)=K(-y)$. Then
\[
\|Lu\|_{C^\al(\R^n)}\le C \|u\|_{C^{\sigma+\al}(\R^n)},
\]
and $C$ is a positive constant depending only on $\al$, $\sigma_0$ and $\Lda$.
\end{lem}
\begin{proof}
First of all, it is clear that
\[
\|Lu\|_{L^\infty(\R^n)}\le C \|u\|_{C^{\sigma+\al}(\R^n)}.
\]
In the following, we are going to estimate the $C^\al$ norm of $Lu$. We first consider that $\sigma+\al\ge 2$, which is the most difficult case. Since
\[
 \begin{split}
  L u(x) &=2\int_{\R^n} (u(x+y)-u(x)-\nabla u(x)y)K(y)\ud y\\
L u(0) &=2\int_{\R^n} (u(y)-u(0)-\nabla u(0)y)K(y)\ud y,
 \end{split}
\]
we have that, for $r=|x|$
\[
 \begin{split}
&\frac{Lu(x)-L(0)}{2}\\&=\int_{B_r} ((u(x+y)-u(x)-\nabla u(x)y)-(u(y)-u(0)-\nabla u(0)y))K(y)\ud y\\
&\quad +\int_{\R^n\setminus B_r} ((u(x+y)-u(x)-\nabla u(x)y)-(u(y)-u(0)-\nabla u(0)y))K(y)\ud y\\
&\quad =I_1+I_2.
 \end{split}
\]
For $I_1$, we have that
\[
 \begin{split}
I_1&=\int_{B_r} (u(x+y)-u(x)-\nabla u(x)y-\frac 12 y^T\nabla^2 u(x)y)K(y)\ud y\\
&\quad -\int_{B_r} (u(y)-u(0)-\nabla u(0)y-\frac 12 y^T\nabla^2 u(0)y)K(y)\ud y\\
&\quad + \frac 12\int_{B_r} ( y^T\nabla^2 u(x)y-y^T\nabla^2 u(0)y)K(y)\ud y,
 \end{split}
\]
and thus
\[
 \begin{split}
|I_1|&\le 2\int_{B_r} \|u\|_{C^{\sigma+\al}(\R^n)} |y|^{\sigma+\al} K(y) \ud y + \|u\|_{C^{\sigma+\al}(\R^n)} r^{\sigma+\al-2}\int_{B_r} |y|^2K(y)\ud y\\
&\le (4\al^{-1}+\Lda)\|u\|_{C^{\sigma+\al}(\R^n)}r^\al.
 \end{split}
\]
For $I_2$, it follows from mean value theorem that
\[
 \begin{split}
|(u(x+y)-u(x)-\nabla u(x)y)-(u(y)-u(0)-\nabla u(0)y)|\le \|u\|_{C^{\sigma+\al}(\R^n)} |x||y|^{\sigma+\al-1}.
 \end{split}
\]
Thus,
\[
 |I_2|\le \|u\|_{C^{\sigma+\al}(\R^n)} |x|\int_{\R^n\setminus B_r}|y|^{\sigma+\al-1}K(y)\ud y\le (1-\al)^{-1}\|u\|_{C^{\sigma+\al}(\R^n)} |x|^\al.
\]
For the case $\sigma+\al<2$, one can prove them similarly and we omit its proof here.
\end{proof}

\begin{lem}\label{lem:derivative-3}
Suppose that $u\in C^{\sigma+\al}(B_2)\cap L^\infty(\R^n)$, $0\le K(y)\le (2-\sigma)\Lda |y|^{-n-\sigma}$, $K(y)=K(-y)$ and $|\nabla K(y)|\le \Lda |y|^{-n-\sigma-1}$. Then
\[
\|Lu\|_{C^\al(B_1)}\le C (\|u\|_{C^{\sigma+\al}(B_2)}+\|u\|_{L^\infty(\R^n)}),
\]
where
\[
L u =\int_{\R^n} \delta u(x,y)K(y)\ud y,
\]
and $C$ is a positive constant depending only on $\al$, $\sigma_0$ and $\Lda$.
\end{lem}
\begin{proof}
Let $\eta\in C_c^\infty(B_{3/2})$ and $\eta\equiv 1$ in $B_{5/4}$. Then
\[
 Lu=L(\eta u)+ L((1-\eta)u).
\]
It follows from Lemma \ref{lem:derivative-2} that
\[
\|L(\eta u)\|_{C^\al(B_1)}\le C \|\eta u\|_{C^{\sigma+\al}(\R^n)} \le C (\| u\|_{C^{\sigma+\al}(B_2)}+\|u\|_{L^\infty(\R^n)}).
\]
For the second term, we have $1-\eta(x)=0$ if $x\in B_1$, and thus
\[
 L((1-\eta)u)(x)=\int_{\R^n} (1-\eta(x+y))u(x+y)K(y)\ud y=\int_{\R^n\setminus B_{5/4}} (1-\eta(y))u(y)K(y-x)\ud y.
\]
The lemma follows immediately since $|\nabla K(y)|\le \Lda |y|^{-n-\sigma-1}$.
\end{proof}

\begin{lem}\label{lem:newapp} Let $v\in C_c^{\sigma+\bar \al}(B_{1/2})$ such that $\|v\|_{C_c^{\sigma+\bar \al}(B_{1/2})}\le 1$, and  $p(x)$ be the Taylor expansion polynomial of $v$ at $x=0$ with degree $[\sigma+\bar\al]$. For every $L\in \mathscr{L}_0(\lda,\Lda, \sigma)$, there exists $P\in C_c^\infty(B_{1/2})$ such that $P(x)=p(x)$ in $B_{1/4}$, $\|P\|_{C^{4}(B_{1/2})} \le C$ and
\[
L P(0) = L v(0),
\]
where $C$ is a positive constant depending only on $n,\lda,\Lda, \sigma_0$ and $\bar\al$.
\end{lem}

\begin{proof} Let $\eta\in C_c^\infty(B_{1/3})$ be such that $\eta\equiv 1$ in $B_{1/4}$. Let  $h(x)\in C^4_c(B_{1/2}\setminus \bar B_{1/3})$ be such that $h(x)=1$ for $B_{11/24}\setminus B_{9/24}$ and $0\le h\le 1$ in $B_{1/2}$. Let $P(x)=\eta(x)p(x)+t\cdot h(x)$, where $t=L(v-\eta p)(0)/Lh(0)$. Then we are left to show that $|t|\le C$, which depends only on $n,\lda,\Lda, \sigma_0$ and $\bar\al$. On one hand, it is clear that
\[
Lh(0)\ge (2-\sigma)C^{-1}.
\]
On the other hand, since $|v(x)-p(x)|\le C |x|^{\sigma+\bar\al}$ for $x\in B_{1/4}$, we have
\[
\begin{split}
|L(v-\eta p)(0)|&=\int_{B_{1/4}}|v(y)-p(y)|K(y)\ud y+\int_{B_{1/2}\setminus B_{1/4}}|v(y)-\eta(y)p(y)|K(y)\ud y\\
&\le C(2-\sigma)\int_{B_{1/4}}|y|^{\sigma+\bar\al-n-\sigma}\ud y +C(2-\sigma)\\
&\le C(2-\sigma),
\end{split}
\]
from which it follows that $|t|\le C$.
\end{proof}

\subsection{Decompositions}

We shall adapt the proofs in \cite{CS11} with delicate decomposition and cut-off arguments indicated in Section \ref{sec:statement} to prove Theorem \ref{thm:E-K-3}.  Recall that we are left to show \eqref{eq:ek v i+1}.

For a function $v$, we denote $v_\rho(x)=\rho^{-(\sigma+\al)}v(\rho x)$. Set
\[
R(x)=\sum_{\ell=0}^i \rho^{-(i-\ell)(\sigma+\al)} v_\ell(\rho^{i-\ell}x).
\]
By \eqref{eq:main},
\[
 \inf_{a\in\mathcal{A}}\{L_a^{(i+1)}R_\rho (x)+\rho^{-(i+1)\al}b_a\}=0 \quad \mbox{in }B_{5/\rho},
\]
where $L^{(i+1)}_{a}$ is the linear operator with kernel $K_a^{(i+1)}\in\mathscr L_2(\lda,\Lda,\sigma)$.
Hence, there exists an $\bar a\in \mathcal{A}$ such that
\be \label{eq:new7}
0\le L^{(i+1)}_{\bar a} R_\rho(0) +\rho^{-(i+1)\al}b_{\bar a}< \rho^{\bar \al-\al}.
\ee

Let $\eta_0(x)=1$ in $B_{1/4}$ and $\eta_0\in C^\infty_c(B_{1/2})$ be a fixed cut-off function. Set
\[
v_\ell(x)=v_\ell\eta_0+v_\ell(1-\eta_0)=:v_\ell^{(1)}+v_\ell^{(2)}.
\]
Let $p_\ell(x)$ be the Taylor expansion polynomial of $v_\ell^{(1)}(x)$ at $x=0$ with degree $[\sigma+\bar\al]$. By Lemma \ref{lem:newapp}, there exists $P_\ell\in C_c^\infty(B_{1/2})$ such that $P_\ell(x)=p_\ell(x)$ in $B_{1/4}$, $\|P_\ell\|_{C^{4}(B_{1/2})} \le c_0$ (a universal constant, independent of $\ell$) and
\be\label{eq:gooddc}
L^{(\ell)}_{\bar a}P_\ell(0)=L^{(\ell)}_{\bar a} v_\ell^{(1)}(0).
\ee
Set
\[
v_\ell=(v_\ell^{(1)}-P_\ell)+ (v_\ell^{(2)}+P_\ell)=:V^{(1)}_{\ell}+V^{(2)}_{\ell}.
\]
We have
\be\label{eq:W1}
\begin{split}
&\|V_\ell^{(1)}\|_{L^\infty(\R^n)}+\|V_\ell^{(2)}\|_{L^\infty(\R^n)}\le c_0+1, \quad V_\ell^{(1)}(0)=0,\\
&V_\ell^{(1)}\in C^{\sigma+\bar\al}_c(B_{1/2}),\quad \|V_\ell^{(1)}\|_{C^{\sigma+\bar\al}(\R^n)}+ \|V_\ell^{(2)}\|_{C^{\sigma+\bar\al}(B_1)}\le 4^4(c_0+1),\\
 & V_\ell^{(1)}=v_\ell-p_\ell\mbox{ in }B_{1/4},\quad  V_\ell^{(2)}=p_\ell\mbox{ in }B_{1/4},\quad |V_\ell^{(1)}(x)|\le 4^4(c_0+1)|x|^{\sigma+\bar\al}\mbox{ in }\R^n.
\end{split}
\ee
Decompose $R(x)$ as
\[
R(x)=R^{(1)}(x)+R^{(2)}(x),
\]
where
\begin{align*}
R^{(1)}(x)&=\sum_{\ell=0}^i  \rho^{-(i-\ell)(\sigma+\al)}  V_\ell^{(1)}(\rho^{i-\ell}x)
\\ R^{(2)}(x)&=\sum_{\ell=0}^i  \rho^{-(i-\ell)(\sigma+\al)}  V_\ell^{(2)}(\rho^{i-\ell}x).
\end{align*}
By change of variables, we have that for each $a\in\mathcal A$,
\be\label{eq:change of variable of V}
\begin{split}
L^{(i+1)}_{a} R_\rho^{(1)}(x)=\sum_{\ell=0}^i \rho^{-(i+1-\ell)\al}(L_a^{(\ell)}V_\ell^{(1)}) (\rho^{i+1-\ell}x),\\
L^{(i+1)}_{a} R_\rho^{(2)}(x)=\sum_{\ell=0}^i \rho^{-(i+1-\ell)\al}(L_a^{(\ell)}V_\ell^{(2)}) (\rho^{i+1-\ell}x).
\end{split}
\ee
By \eqref{eq:new7} and \eqref{eq:gooddc}, we have
\be \label{eq:d1}
\begin{split}
L^{(i+1)}_{\bar a} R_\rho^{(1)}(0)&=0,\\
 0\le  L^{(i+1)}_{\bar a} R_\rho^{(2)}(0)+\rho^{-(i+1)\al}b_{\bar a}&=L^{(i+1)}_{\bar a} R_\rho(0)+\rho^{-(i+1)\al}b_{\bar a}\le \rho^{\bar \al-\al}.
\end{split}
\ee
It follows from Lemma \ref{lem:derivative-2}, \eqref{eq:change of variable of V}, \eqref{eq:d1} and \eqref{eq:W1} that
\begin{align}
|L^{(i+1)}_{\bar a} R_\rho^{(1)}(x)|&=|L^{(i+1)}_{\bar a} R^{(1)}_\rho(x)-L^{(i+1)}_{\bar a} R^{(1)}_\rho(0)| \nonumber  \\
& \le \sum_{\ell=0}^i \rho^{-(i+1-\ell)\al}|(L_a^{(\ell)}V_\ell^{(1)}) (\rho^{i+1-\ell}x)- (L_a^{(\ell)}V_\ell^{(1)})(0)|\nonumber\\&
\le  C |x|^{\bar\al} \sum_{\ell=0}^i  \rho^{(i+1-\ell)(\bar \al-\al)}\|V_\ell^{(1)}\|_{C^{\sigma+\bar\al}(\R^n)}\nonumber \\&
\le  C |x|^{\bar\al} \rho^{\bar \al-\al}\sum_{\ell=0}^\infty  \rho^{\ell(\bar \al-\al)}\nonumber \\&
\le  C \rho^{\bar \al-\al}|x|^{\bar\al} \quad \mbox{ for }x\in \R^n.
\label{eq:r1}
\end{align}
Similarly, it follows from Lemma \ref{lem:derivative-3}, \eqref{eq:change of variable of V}and \eqref{eq:W1} that
\begin{align}
|L^{(i+1)}_{\bar a} R^{(2)}_\rho(x)-L^{(i+1)}_{\bar a} R^{(2)}_\rho(0)| & \le \sum_{\ell=0}^i \rho^{-(i+1-\ell)\al}|(L_a^{(\ell)}V_\ell^{(2)}) (\rho^{i+1-\ell}x)- (L_a^{(\ell)}V_\ell^{(2)})(0)|\nonumber\\&
\le  C |x|^{\bar\al} \sum_{\ell=0}^i  \rho^{(i+1-\ell)(\bar \al-\al)}(\|V_\ell^{(2)}\|_{C^{\sigma+\bar\al}(B_1)}+\|V_\ell^{(2)}\|_{L^{\infty}(\R^n)})\nonumber \\&
\le  C \rho^{\bar \al-\al}|x|^{\bar\al} \quad \mbox{ for }x\in B_5.
\label{eq:r2}
\end{align}
Thus, by \eqref{eq:d1}, we have
\begin{align}
         \label{eq:new9}
|L^{(i+1)}_{\bar a} R_\rho^{(2)}(x)+\rho^{-(i+1)\al}b_{\bar a}|\le C \rho^{\bar \al-\al} (|x|^{\bar\al}+1)\quad  \mbox{ for }x\in B_5.
\end{align}
Let
\[
\tilde v_{i+1}=v_{i+1}+R^{(1)}_\rho.
\]
Hence, the equation of \eqref{eq:main} involving $v_{i+1}$ is 
\[
\inf_{a}\{ L^{(i+1)}_a (v_{i+1}+R_\rho)+\rho^{-(i+1)\al}b_a\}=0,
\]
which is equivalent to
\be \label{eq:new}
\inf_{a}\{L^{(i+1)}_a (\tilde v_{i+1}+R^{(2)}_\rho)+\rho^{-(i+1)\al}b_{a}\}=0\quad\mbox{in }B_5.
\ee
It follows from \eqref{eq:r1} and \eqref{eq:new9} that
\be \label{eq:lower}
\begin{split}
L^{(i+1)}_{\bar a}  v_{i+1}(x)&\ge -C \rho^{\bar \al-\al} \quad \mbox{in }B_5,\\
L^{(i+1)}_{\bar a} \tilde v_{i+1}(x)&\ge -C \rho^{\bar \al-\al} \quad \mbox{in }B_5,
\end{split}
\ee
where $C$ is a universal positive constant.

\subsection{$C^\sigma$ estimates}
Define the maximal operators
\[
\begin{split}
\mathcal M^+_0 u(x)=\sup_{K\in \mathscr L_0(\lda,\Lda,\sigma)}\int_{\R^n}\delta u(x,y)K(y)\ud y,\\
\mathcal M^+_2 u(x)=\sup_{K\in \mathscr L_2(\lda,\Lda,\sigma)}\int_{\R^n}\delta u(x,y)K(y)\ud y.
\end{split}
\]
And one can define the extremal operators $\mathcal M^-_0$ and $\mathcal M^-_2$ similarly. Let $\eta_1\in C_c^\infty(B_{4})$ be a smooth cut-off function such that $\eta_1\equiv 1$ in $B_{3}$. We write \eqref{eq:new} as
\be \label{eq:aux0}
\inf_{a\in \mathcal{A}} \{L^{(i+1)}_a \tilde v_{i+1}+h_a(x)+\rho^{-(i+1)\al}b_{a}\}=0 \quad \mbox{in }B_3,
\ee
where
\[
h_a(x):= \eta_1(x) L^{(i+1)}_a R^{(2)}_\rho(x).
\]

\begin{lem}\label{lem:6.1} Let $K$ be a symmetric kernel satisfying $0\le K(y)\le (2-\sigma)\Lda |y|^{-n-\sigma}$. Then for every bump function $\eta$ such that
\begin{align*}
0&\le \eta(x)\le 1\quad \mbox{in }\R^n,\\
\eta(x)&=\eta(-x)\quad \mbox{in }\R^n,\\
\eta(x)&=0 \quad \mbox{in }\R^n\setminus B_{3/2},
\end{align*}
we have
\[
\mathcal M^+_2\left(\int_{\R^n}\delta \tilde v_{i+1} (x,y)K(y)\eta(y)\,\ud y\right)\ge -C\rho^{2-\al} \quad \mbox{in }B_{3/2}.
\]
\end{lem}

\begin{proof} Let $\phi_k$ be the $L^1$ function $\phi_k=\chi_{\R^n\setminus B_{1/k}} K(y)\eta(y)$, where $\chi_E$ is the characteristic function of a set $E$.
For every $a\in\mathcal A$, we know from \eqref{eq:aux0} that
\[
L^{(i+1)}_a \tilde v_{i+1}(x)+h_a(x)+\rho^{-(i+1)\al}b_{a}\ge 0 \quad \forall ~x\in B_{3}.
\] 
It follows that for all $x\in B_{3/2}$,
\[
\begin{split}
0&\le (L^{(i+1)}_a \tilde v_{i+1}+h_a+\rho^{-(i+1)\al}b_{a})* \phi_k (x)\\
&\le L^{(i+1)}_a (\tilde v_{i+1}* \phi_k)(x)+h_a* \phi_k(x)+\rho^{-(i+1)\al}b_{a}\|\phi_k\|_{L^1}.
\end{split}
\]
It also follows from \eqref{eq:aux0} that 
\[
\inf_{a\in\mathcal A}\{ \|\phi_k\|_{L^1}( L^{(i+1)}_a \tilde v_{i+1}(x)+h_a(x)+\rho^{-(i+1)\al}b_{a}) \}=0\quad\forall\ x\in B_3.
\]
This implies that for all $x\in B_{3/2}$,
\[
\sup_{a\in\mathcal A}L^{(i+1)}_a (\tilde v_{i+1}* \phi_k -\|\phi_k\|_{L_1}\tilde v_{i+1})(x)+\sup_{a\in\mathcal A} \{h_a* \phi_k(x)-\|\phi_k\|_{L_1}h_a(x)\}\ge 0.
\]
For any $x\in B_{3/2}$, any $a\in\mathcal A$, by using \eqref{eq:change of variable of V} and change of variables we have
\begin{align*}
&2|h_a* \phi_k(x)-\|\phi_k\|_{L_1}h_a(x)|\\&\le |\int_{B_{3/2}\setminus B_{1/k}}\delta (L_a^{(i+1)} R^{(2)}_\rho) (x,y)K(y)\eta(y)\,\ud y|\\&
\le\sum_{\ell=0}^i \rho^{(\ell-1-i)\al}\int_{B_{3/2}\setminus B_{1/k}}|\delta (L_a^{(\ell)}V_{\ell}^{(2)})(\rho^{i+1-\ell}x, \rho^{i+1-\ell}y)|K(y)\eta(y)\,\ud y\\&
\le\sum_{\ell=0}^i \rho^{(i+1-\ell)(\sigma-\al)}\int_{B_{3\rho^{i+1-\ell}/2}\setminus B_{\rho^{i+1-\ell}/k}}|\delta (L_a^{(\ell)}V_{\ell}^{(2)})(\rho^{i+1-\ell}x, y)|K^{-(i+1-\ell)}(y)\,\ud y\\&
\le\sum_{\ell=0}^i \rho^{(i+1-\ell)(\sigma-\al)}\int_{B_{3\rho^{i+1-\ell}/2}}\|L_a^{(\ell)}V_{\ell}^{(2)}\|_{C^2(B_{1/8})}|y|^2K^{-(i+1-\ell)}(y)\,\ud y\\&
\le   \sum_{\ell=0}^i \rho^{(i+1-\ell)(\sigma-\al)}\|L_a V_\ell^{(2)}\|_{C^2(B_{1/8})} \int_{B_{3\rho^{i+1-\ell}/2}} \frac{\Lda (2-\sigma)}{|y|^{n+\sigma-2}}\,\ud y\\&
\le C   \sum_{\ell=0}^i \rho^{(i+1-\ell)(\sigma-\al)}(\|V_\ell^{(2)}\|_{C^4(B_{1/4})}+\|V_\ell^{(2)}\|_{L^\infty(\R^n)}) \rho^{(i+1-\ell)(2-\sigma)} \\&
\le C\rho^{2-\al}(\|V_\ell^{(2)}\|_{C^4(B_{1/4})}+\|V_\ell^{(2)}\|_{L^\infty(\R^n)}) \sum_{\ell=0}^\infty \rho^{\ell(2-\al)} \\&
\le C \rho^{2-\al},
\end{align*}
where $K^{-(i+1-\ell)}(y)=\rho^{-(i+1-\ell)(n+\sigma)}K(\rho^{-(i+1-\ell)} y)$, and Lemma \ref{lem:derivative-1} was used since
$V_\ell^{(2)}(x)=p_{\ell}(x)$ in $B_{1/4}$.
Consequently,
\[
\mathcal M^+_2(\tilde v_{i+1}* \phi_k -\|\phi_k\|_{L_1}\tilde v_{i+1})(x) \ge -C \rho^{2-\al}.
\]
The result follows from Lemma 5 in \cite{CS10} by taking the limit as $k\to\infty$.
\end{proof}

\begin{lem}\label{lem:7.1} Let $K$ be a symmetric kernel satisfying $0\le K(y)\le (2-\sigma)\Lda |y|^{-n-\sigma}$.
Then for every smooth bump function $\eta$ such that
\begin{align*}
0\le \eta(x)\le 1\quad \mbox{in }\R^n, \qquad  \eta(x)&=\eta(-x)\quad \mbox{in }\R^n, \\
\eta(x)=0 \quad \mbox{in }\R^n\setminus B_{4/5} ,\qquad \eta(x)&=1\quad \mbox{in }B_{3/4},
\end{align*}
we have
\[
\mathcal M^+_2\left(\eta(x)\int_{B_{1}}\delta \tilde v_{i+1}(x,y)K(y)\,\ud y\right)\ge -C(\rho^{\bar \al-\al}+\frac{1}{M}) \quad \mbox{in }B_{3/5}.
\]
\end{lem}

\begin{proof}  Define
\[
Tv(x)=\int_{B_{1}} \delta v(x,y)K(y)\,\ud y.
\]
It follows from Lemma \ref{lem:6.1} that
\be \label{eq:lo1}
\mathcal M^+_2(T \tilde v_{i+1})(x) \ge -C\rho^{2-\al}\quad\mbox{in }B_{3/2}.
\ee
Let $\bar L$ be any operator with kernel $\bar K\in\mathscr{L}_2(\lda,\Lda,\sigma)$. For $x\in B_{3/5}$, we have
\begin{align}
\bar L(\eta T\tilde v_{i+1})(x)&= \int_{\R^n} \delta(T\tilde v_{i+1})(x,y)\bar K(y)\,\ud y- \int_{\R^n} \delta((1-\eta)T\tilde v_{i+1})(x,y)\bar K(y)\,\ud y \nonumber\\&
\ge \bar L(T\tilde v_{i+1})(x)-2 \int_{\R^n} (1-\eta(x-y)) T\tilde v_{i+1}(x-y) \bar K(y) \,\ud y.\label{eq:L bar esti}
\end{align}
Now we estimate the second term in the last inequality. Recall that $\tilde v_{i+1}=v_{i+1}+R^{(1)}_\rho$. It is clear that
\begin{align}
&\int_{\R^n}  T v_{i+1}(x-y) (1-\eta(x-y)) \bar K(y)\,\ud y \nonumber\\&
= \int_{\R^n}   v_{i+1}(x-y)T ((1-\eta(x-\cdot)) \bar K(\cdot))(y)\,\ud y
\le C\|v_{i+1}\|_{L^\infty}\le C/M.\nonumber
\end{align}
By change of variables, we have for all $x\in \R^n$,
\begin{align}
|T R^{(1)}_\rho(x)| &=|\int_{B_1} \delta R^{(1)}_\rho(x,y)K (y)\,\ud y|
\nonumber  \\&
=| \sum_{\ell=0}^i \int_{B_{\rho^{i+1-\ell}}}  \rho^{-(i+1-\ell)\al} \delta V_\ell^{(1)}(\rho^{i+1-\ell} x, y) K^{-(i+1-\ell)}(y) \,\ud y |,\nonumber 
\end{align}
where $K^{-(i+1-\ell)}(y)=\rho^{-(i+1-\ell)(n+\sigma)}K(\rho^{-(i+1-\ell)} y)$.

By triangle inequality, we have
\begin{align}
&|T R^{(1)}_\rho(x)|\nonumber \\ &
\le \sum_{l=0}^i \rho^{-(i+1-\ell)\al}  | \int_{B_{\rho^{i+1-\ell}}}  (\delta V_\ell^{(1)}(\rho^{i+1-\ell} x, y)-\delta V_\ell^{(1)}(0, y)) K^{-(i+1-\ell)}(y) \,\ud y |\nonumber \\&
\quad
 + \sum_{l=0}^i \rho^{-(i+1-\ell)\al}  | \int_{B_{\rho^{i+1-\ell}}}  \delta V_\ell^{(1)}(0, y) K^{-(i+1-\ell)}(y) \,\ud y |  \nonumber \\&
\le  C\sum_{\ell=0}^i \|V_\ell^{(1)} \|_{C^{\sigma+\bar \al}(\R^n)} |x|^{\bar \al}  \rho^{(i+1-\ell)(\bar \al-\al)} \nonumber \\&
\quad
+ C\sum_{l=0}^i \rho^{-(i+1-\ell)\al}  \int_{B_{\rho^{i+1-\ell}}} \frac{(2-\sigma)\Lda|y|^{\sigma+\bar\al} }{|y|^{n+\sigma}}\,\ud \zeta\nonumber\\&
\le C\rho^{\bar \al-\al}(1+|x|^{\bar \al})  \quad\mbox{ for all }x\in \R^n,
\label{eq:11}
    \end{align}
where we used Lemma \ref{lem:derivative-2} and \eqref{eq:W1} in the second inequality.

It follows that for $x\in B_{3/5}$,
\begin{align}
&\int_{\R^n} (1-\eta(x-y)) T R^{(1)}_{\rho}(x-y) \bar K(y) \,\ud y\nonumber\\&= \int_{\R^n} (1-\eta(y)) T R^{(1)}_{\rho}(y) \bar K(x-y) \,\ud y\nonumber\\&
=\int_{\R^n\setminus B_{3/4}} (1-\eta(y)) T R^{(1)}_{\rho}(y) \bar K(x-y) \,\ud y\nonumber\\&
\le C\rho^{\bar \al-\al}\int_{|y|>1/64}\frac{(2-\sigma)}{|y|^{n+\sigma-\bar \al}}\le  C\rho^{\bar \al-\al},\label{eq:estimate of T}
\end{align}
where we used that $\sigma\ge\sigma_0>1>\bar\al$. Taking the supremum of all $\bar K$ in $\mathscr{L}_2(\lda,\Lda,\sigma)$ in \eqref{eq:L bar esti} and using \eqref{eq:lo1}, we complete the proof.
\end{proof}

\begin{lem}\label{lem:6.3} We have
\[
|L^{(i+1)}_{\bar a}  v_{i+1}(x)| \le C(\rho^{\bar \al-\al}+\frac{1}{M})\quad \mbox{in }B_{1/2}.
\]
\end{lem}

\begin{proof} Let $\eta_1(x)\ge 0$ be a smooth cutoff function in $B_2$ with $\eta_1\equiv 1$ in $B_{1}$. Then
\begin{align*}
\int_{\R^n} L^{(i+1)}_{\bar a} v_{i+1}\eta_1 =\int_{\R^n} v_{i+1}L^{(i+1)}_{\bar a}\eta_1 \le C\|v_{i+1}\|_{L^\infty(\R^n)}\le C/M.
\end{align*}
By \eqref{eq:lower}, $L^{(i+1)}_{\bar a} v_{i+1}\ge -C\rho^{\bar \al-\al}$ in $B_4$, we have
\begin{align*}
\int_{B_{1}} |L^{(i+1)}_{\bar a} v_{i+1}| \le   C(\rho^{\bar \al-\al}+\frac 1M).
\end{align*}
Let 
\[
T^{(i+1)}_{\bar a} v=\int_{B_1}\delta v(x,y) K^{(i+1)}_{\bar a}(y)\,\ud y.
\]
 It is easy to see that
\[
\int_{B_{1}} |T^{(i+1)}_{\bar a} v_{i+1}| \le \int_{B_{1}} |L^{(i+1)}_{\bar a} v_{i+1}|+  C\|v_{i+1}\|_{L^\infty}\le C(\rho^{\bar \al-\al}+\frac 1M).
\]
It follows from \eqref{eq:11} that for all $x\in \R^n$, 
\be\label{eq:esti in B of T}
|T^{(i+1)}_{\bar a} R^{(1)}_\rho(x)|\le C\rho^{\bar \al-\al}(1+|x|^{\bar \al}) .
\ee
Since $\tilde v_{i+1}=v_{i+1}+R^{(1)}_\rho$, we  obtain
\be \label{eq:L1}
\int_{B_{1}}| T^{(i+1)}_{\bar a}\tilde v_{i+1}|
 \le  C(\rho^{\bar \al-\al}+\frac 1M).
\ee
Let $\eta$ be the smooth cut-off function in Lemma \ref{lem:7.1}, and denote $v(x):=\eta(x)T^{(i+1)}_{\bar a} \tilde v_{i+1}(x)$. It follows from Lemma \ref{lem:7.1} that
\[
\mathcal M^+_2  v(x) \ge -C(\rho^{\bar \al-\al}+\frac 1M)\quad\mbox{in }B_{3/5}.
\]
It follows from  \eqref{eq:L1} and Theorem 5.1 in \cite{CS11} that $v\le C(\rho^{\bar \al-\al}+\frac 1M)$ in $B_{1/2}$.  But $v=T^{(i+1)}_{\bar a} \tilde v_{i+1}$ in $B_{1/2}$, so we have proved that 
\[
 T^{(i+1)}_{\bar a} \tilde v_{i+1} \le C(\rho^{\bar \al-\al}+\frac 1M)\quad \mbox{in} B_{1/2}.
 \]
 By \eqref{eq:esti in B of T}, we have $ T^{(i+1)}_{\bar a} v_{i+1}(x) \le C(\rho^{\bar \al-\al}+\frac 1M)$ in $B_{1/2}$, and thus,  
\[
 L^{(i+1)}_{\bar a} v_{i+1}(x) \le C(\rho^{\bar \al-\al}+\frac 1M)\quad \mbox{in } B_{1/2}.
\]
 We complete the proof together with \eqref{eq:lower}.
\end{proof}

\begin{lem} \label{lem:7.2}There is a universal constant $C$ such that for every operator $L$ with a symmetric kernel $K$ satisfying $0\le K(y)\le (2-\sigma) \Lda|y|^{n+\sigma}$, we have
\[
|L v_{i+1}(x)| \le C(\rho^{\bar \al-\al}+\frac{1}{M}) \quad \mbox{in }B_{1}.
\]
\end{lem}

\begin{proof} We will prove the estimate in $B_{1/6}$, and the general estimate follows from scaling and translation arguments. By Lemma \ref{lem:6.3} we have
\[
\|L^{(i+1)}_{\bar a} v_{i+1}\|_{L^2(B_{1/2})}  \le C(\rho^{\bar \al-\al}+\frac{1}{M}).
\]
Note that $\|v_{i+1}\|_{L^1(\R^n, 1/(1+|y|^{n+\sigma}))}\le C\|v_{i+1}\|_{L^\infty}\le C/M$. From Theorem 4.3 of \cite{CS11}, we have $L^2$ estimate for every linear operator $L$ with kernel $K\in \mathscr{L}_0(\lda,\Lda,\sigma)$,
\[
\|L v_{i+1}\|_{L^2(B_{1/3})} \le C(\rho^{\bar \al-\al}+\frac{1}{M}).
\]
We split the integral of $Lv_{i+1}$ as
\[
Lv_{i+1}(x)=\int_{B_{1}} +\int_{B_{1}^c} \delta v_{i+1}(x,y) K(y)\,\ud y.
\]
It is clear that
\[
|\int_{B_{1}^c} \delta v_{i+1}(x,y) K(y)\,\ud y| \le C\|v_{i+1}\|_{L^\infty}\le C/M.
\] Hence,  we have $L^2$ estimates for the first one
\[
\left\|\int_{B_{1}} \delta v_{i+1}(x,y) K(y)\,\ud y\right\|_{L^2(B_{1/3})}\le C(\rho^{\bar \al-\al}+\frac{1}{M}).
\]
It follows from \eqref{eq:11} that
\be\label{eq:Lw est}
\begin{split}
|\int_{B_{1}} \delta R^{(1)}_\rho(x,y) K(y)\,\ud y|\le C\rho^{\bar \al-\al}\quad\mbox{for }x\in B_1.
\end{split}
\ee
By triangle inequality, we have
\be \label{eq:L2}
\left\|\int_{B_{1}} \delta \tilde v_{i+1}(x,y) K(y)\,\ud y\right\|_{L^2(B_{1/3})}\le C(\rho^{\bar \al-\al}+\frac{1}{M}).
\ee
For a smooth cut-off function $c(x)\in C_c^\infty(B_{1/3})$, $c(x)=c(-x)$, and $c(x)=1$ in $B_{1/4}$, we define
\[
v(x):=c(x)\int_{B_{1}}\delta \tilde v_{i+1}(x,y)K(y)\,\ud y.
\]
It follows from the proof of Lemma \ref{lem:7.1} that $\mathcal M^+_2 v\ge -C(1/M+\rho^{\bar \al-\al})$ in $B_{1/5}$. By \eqref{eq:L2} and Theorem 5.1 in \cite{CS11} we have
$v \le C(1/M+\rho^{\bar \al-\al})$ in $B_{1/6}$, and thus
\[
\int_{B_{1}}\delta \tilde v_{i+1}(x,y)K(y)\,\ud y \le C(1/M+\rho^{\bar \al-\al})\quad\mbox{in }B_{1/6}.
\]
Since \eqref{eq:Lw est} holds for $x\in B_{1}$, we have that
\[
\int_{B_{1}}\delta v_{i+1}(x,y)K(y)\,\ud y \le C(1/M+\rho^{\bar \al-\al})\quad\mbox{in }B_{1/6}.
\]
Consequently,
\[
Lv_{i+1}\le C(1/M+\rho^{\bar \al-\al})\quad\mbox{in }B_{1/6}.
\]

Consider the kernel
\[
K_d= \frac{2}{\lda} K^{(i+1)}_{\bar a}-\frac{1}{\Lda} K
\]
and the corresponding linear operator $L_d$, where $0\le K\le (2-\sigma)\Lda |y|^{-n-\sigma}$. The kernel $K_d$ satisfies the ellipticity condition $(2-\sigma)|y|^{-n-\sigma}\le K_d(y)\le (2-\sigma)(2\Lda/\lda)|y|^{-n-\sigma}$. The same proof as above yields that
\[
L_d  v_{i+1}\le C(1/M+\rho^{\bar \al-\al})\quad\mbox{in }B_{1/6}.
 \]Since $L^{(i+1)}_{\bar a}  v_{i+1}$ is lower bounded by \eqref{eq:lower}, we obtain a bound from below for $L$ in $B_{1/6}$
\[
L v_{i+1}=2\frac{\Lda}{\lda} L^{(i+1)}_{\bar a} v_{i+1}-\Lda L_d v_{i+1}\ge -C(1/M+\rho^{\bar \al-\al})\quad\mbox{in }B_{1/6}.
\]
Similarly, if we consider $\tilde K_d=\frac{2}{\lda} K^{(i+1)}_{\bar a}+\frac{1}{\Lda} K$, we obtain that $L v_{i+1}\le C(1/M+\rho^{\bar \al-\al})$. 
In conclusion, we obtained that $|L v_{i+1}|\le C(1/M+\rho^{\bar \al-\al})$ in  $B_{1/6}$.
\end{proof}

The above lemma immediately gives

\begin{cor}\label{cor:7.3} $\mathcal M^+_0  v_{i+1}$ and $\mathcal M^-_0 v_{i+1}$ are bounded by $C(\rho^{\bar \al-\al}+\frac{1}{M})$ in $B_{1}$.
In particular,
\be\label{eq:grad est u}
\|\nabla  v_{i+1}\|_{L^\infty(B_{1/2})}\le C(\rho^{\bar \al-\al}+\frac{1}{M}),
\ee
and consequently,
\be\label{eq:grad est}
\|\nabla \tilde v_{i+1}\|_{L^\infty(B_{1/2})}\le C(\rho^{\bar \al-\al}+\frac{1}{M}).
\ee
\end{cor}
\begin{proof}
The first conclusion is clear, from which \eqref{eq:grad est u} also follows immediately since  $\sigma\ge\sigma_0>1$. To prove \eqref{eq:grad est}, we notice that $V_\ell^{(1)}=v^{(1)}_\ell-P_\ell \in C^{\sigma+\bar \al}_c(B_{1/2})$, and $V_\ell^{(1)}=v^{(1)}_\ell-p_\ell $ in $B_{1/4}$ where $p_\ell$ is the Taylor expansion polynomial of $v^{(1)}_\ell$ at $x=0$ with degree $[\sigma+\bar\al]$. Hence, $|\nabla V_\ell^{(1)}(x)|\le C|x|^{\sigma+\bar\al-1}$ in $B_{1/2}$. Thus, for all $x\in B_{1/2}$,
\begin{align*}
|\nabla  R^{(1)}_\rho(x)|&= |\nabla \sum_{\ell=0}^i \rho^{-(i+1-\ell)(\sigma+\al)} V^{(1)}_\ell(\rho^{i+1-\ell}x)|\\&
\le C\sum_{\ell=0}^i \rho^{-(i+1-\ell)(\sigma+\al-1)} |\rho^{i+1-\ell}x|^{\sigma+\bar\al-1}
\le C\rho^{\bar \al-\al} .
\end{align*}
Thus, \eqref{eq:grad est} follows immediately.
\end{proof}

\begin{thm} \label{thm:sigma} We have
\[
\int_{\R^n}|\delta v_{i+1}(x,y)| \frac{(2-\sigma)}{|y|^{n+\sigma}}\,\ud y \le C(\rho^{\bar \al-\al} +\frac{1}{M})\quad \mbox{in }B_{1}.
\]
\end{thm}

\begin{proof} Given Lemma \ref{lem:7.2} and Corollary \ref{cor:7.3},  it follows from the same proof as that of Theorem 7.4 in \cite{CS11}.
\end{proof}

\subsection{$C^{\sigma+\bar\al}$ estimates}

For brevity, we write 
\[
u=v_{i+1} \quad\mbox{and}\quad \tilde u =\tilde v_{i+1}.
\]
 in this subsection.

Let $\eta$ be a bump function as in Lemma \ref{lem:7.1}. For each measurable set $A$ with $-A=A$, we write
\[
w_A(x)=\eta(x)\int_{B_{1}} (\delta \tilde u (x,y)-\delta \tilde u(0,y) )K_A(y)\,\ud y,
\]
where
\[
K_A(y)=\frac{(2-\sigma)}{|y|^{n+\sigma}} \chi_A(y).
\]
For $x\in B_{1}$, by Lemma \ref{lem:derivative-2} and change of variables, we have
\begin{align}
&|\int_{B_{1}} (\delta R^{(1)}_\rho(x,y) -\delta R^{(1)}_\rho(0,y))K_A(y)\,\ud y|\nonumber\\&
=|\sum_{\ell=0}^i \rho^{-(i+1-\ell)(\sigma+\al)}\int_{B_{1}}(\delta V_\ell^{(1)}(\rho^{i+1-\ell}x, \rho^{i+1-\ell}y)-\delta V_\ell^{(1)}(0, \rho^{i+1-\ell}y))K_A(y)\,\ud y|\nonumber\\&
=|\sum_{\ell=0}^i \rho^{-(i+1-\ell)\al}\int_{B_{\rho^{i+1-\ell}}}(\delta V_\ell^{(1)}(\rho^{i+1-\ell}x, y)-\delta V_\ell^{(1)}(0, y))K^{(\ell-1-i)}_A(y)\,\ud y|\nonumber\\&
\le \sum_{\ell=0}^i \rho^{-(i+1-\ell)\al}\|V^{(1)}_\ell\|_{C^{\sigma+\bar\al}(\R^n)}\rho^{(i+1-\ell)\bar\al}|x|^{\bar\al}\nonumber\\&
\le C\rho^{\bar \al-\al}|x|^{\bar \al}.\label{eq:r1 holder}
\end{align}
Then it follows from Theorem \ref{thm:sigma} that 
\be\label{eq:bound wa}
|w_A| \le C(\rho^{\bar \al-\al}+1/M)\quad\mbox{in}\quad\R^n. 
\ee
Also, it follows from Lemma \ref{lem:7.2} as well as  \eqref{eq:11} that
\[
 |\int_{B_{1}} \delta \tilde u(0,y)K_A(y)\,\ud y|\le C(\rho^{\bar \al-\al}+1/M).
\]
Together with Lemma \ref{lem:7.1},  we have
\[
\mathcal M^+_2 w_A \ge -C(\rho^{\bar \al-\al}+1/M) \quad \mbox{in }B_{3/5} \mbox{ uniformly in } A.
\]
As in \cite{CS11}, we define
\[
N^+(x):=\sup_{A} w_A(x)=\eta(x)\int_{B_{1}} (\delta \tilde u(x,y)- \delta \tilde u(0,y))^+\frac{(2-\sigma)}{|y|^{n+\sigma}}\,\ud y,
\]
\[
N^-(x):=\sup_{A} -w_A(x)=\eta(x)\int_{B_{1}} (\delta \tilde u(x,y)- \delta \tilde u(0,y))^-\frac{(2-\sigma)}{|y|^{n+\sigma}}\,\ud y.
\]

\begin{lem}\label{lem:9.1} For all $x\in B_{1/4}$, we have
\[
\frac{\lda}{\Lda} N^-(x)-C(\rho^{\bar \al-\al}+1/M) |x| \le N^+(x) \le \frac{\Lda}{\lda} N^-(x)+C(\rho^{\bar \al-\al}+1/M) |x|.
\]
\end{lem}

\begin{proof} For some $x\in B_{1/2}$, let $\tilde u_x(z)=\tilde u(x+z)$. It follows from \eqref{eq:aux0} that
\[
\mathcal M^+_2(\tilde u_x-\tilde u)(0)\ge -\sup_{a}(h_a(x)-h_a(0)), \quad \mathcal M^-_2(\tilde u_x-\tilde u)(0)\le \sup_{a}(h_a(0)-h_a(x)).
\]
Note that for $x\in B_3$,
\begin{align*}
h_a(x)= L_a^{(i+1)} R_\rho^{(2)}(x)=\sum_{\ell=0}^i \rho^{-(i+1-\ell)\al}(L_a^{(\ell)}V_\ell^{(2)}) (\rho^{i+1-\ell}x)\end{align*}
and thus  for $\rho x\in B_{1/4}$
\begin{align*}
|h_a(x)-h_a(0)|&
=| \sum_{\ell=0}^i \rho^{-(i+1-\ell)\al} ( L_a ^{(\ell)} V_\ell^{(2)}(\rho^{(i+1-\ell)} x) - L_a ^{(\ell)} V_\ell^{(2)}(0) )|\\&
\le C \sum_{\ell=0}^i \rho^{-(i+1-\ell)\al}  (\|V_\ell^{(2)}\|_{C^4(B_{1/2})}+\|V_\ell^{(2)}\|_{L^\infty(\R^n)})|\rho ^{(i+1-\ell)} x|\\&
\le C\rho^{1-\al}  \sum_{\ell=0}^\infty \rho^{\ell(1-\al)} |x|,
\end{align*}
where Lemma \ref{lem:derivative-1} was used in the first inequality.
Hence, we have
\be\label{eq:pucci}
\mathcal M^+_2(\tilde u_x-\tilde u)(0)\ge -C\rho^{1-\al} |x|, \quad \mathcal M^-_2(\tilde u_x-\tilde u)(0)\le C\rho^{1-\al}|x|.
\ee

For every kernel $K\in \mathscr{L}_2(\lda,\Lda,\sigma)$, we have
\begin{align*}
L(\tilde u_x-\tilde u)(0)&=\int_{\R^n} (\delta \tilde u(x,y)-\delta \tilde u(0,y))K(y)\,\ud y\\&
=\int_{B_{1}} (\delta \tilde u(x,y)-\delta \tilde u(0,y))K(y)\,\ud y\\&
\quad + \int_{\R^n \setminus B_{1}} (\delta \tilde u(x,y)-\delta \tilde u(0,y))K(y)\,\ud y.
\end{align*}
Now we estimate the second term of right hand side: for $x\in B_{1/4}$

\be\label{eq:error delta}
\begin{split}
& \frac 12\int_{\R^n \setminus B_{1}} (\delta \tilde u(x,y)-\delta \tilde u(0,y))K(y)\,\ud y\\&
=\int_{\R^n} \tilde u(y)(K(y-x) \chi_{B_{1}^c}(y-x)\\
&\quad-K(y)\chi_{B_{1}^c}(y))\,\ud y
-(\tilde u(x)-\tilde u(0))\int_{\R^n\setminus B_{1}}K(y)\,\ud y\\&
\le \int_{\R^n\setminus B_{1+|x|}} | \tilde u(y)||K(y-x)-K(y)|\,\ud y\\
&\quad+\|\tilde u\|_{L^\infty(B_{1+|x|})}\int_{B_{1+|x|} \setminus B_{1-|x|}}K(y)\,\ud y+C(\rho^{\bar \al-\al}+\frac{1}{M})|x|\\&
\le C(\rho^{\bar \al-\al}+\frac{1}{M}) |x|,
\end{split}
\ee
where in the first inequality we have used \eqref{eq:grad est}, and in the last we used that $|\nabla K(y)|\le(2-\sigma)\Lda|y|^{-n-\sigma-1} $ and
\[
\begin{split}
 |\tilde u(y)| &\le \|v_{i+1}\|_{L^\infty(\R^n)} + |R^{(1)}_\rho(y)| \\&
  \le \frac{1}{M}+C\sum_{\ell=0}^i \rho^{-(i+1-\ell)(\sigma+\al)}|V_\ell^{(1)}(\rho^{i+1-\ell}y)|\\&
 \le \frac{1}{M}+C\sum_{\ell=0}^i \rho^{-(i+1-\ell)(\sigma+\al)}|\rho^{i+1-\ell} y|^{\sigma+\bar\al}\\&
 \le \frac{1}{M}+C\rho^{\bar \al-\al} |y|^ {\sigma+\bar\al}.
\end{split}
\]
Therefore, for every kernel $K\in \mathscr{L}_2(\lda,\Lda,\sigma)$, we have
\[
\int_{\R^n} (\delta \tilde u(x,y)-\delta \tilde u(0,y))K(y)\,\ud y \le \int_{B_{1}} (\delta \tilde u(x,y)-\delta \tilde u(0,y))K(y)\,\ud y+C
(\rho^{\bar \al-\al}+\frac{1}{M})|x|.
\]
Taking the supremum and using \eqref{eq:pucci}, we obtain
\[
-C\rho^{1-\al}|x|\le \mathcal M^+_2(\tilde u_x-\tilde u)\le \sup_{K}   \int_{B_{1}} (\delta \tilde u(x,y)-\delta \tilde u(0,y))K(y)\,\ud y +C(\rho^{\bar \al-\al}+\frac{1}{M})|x|.
\]
In particular, if we take the supremum over all kernels $K\in\mathscr{L}_0(\lda,\Lda,\sigma)$, we still have
\[
\sup_{\frac{\lda (2-\sigma)}{|y|^{n+\sigma}}\le K \le\frac{\Lda (2-\sigma)}{|y|^{n+\sigma}}}   \int_{B_{1}} (\delta \tilde u(x,y)-\delta \tilde u(0,y))K(y)\,\ud y \ge -C(\rho^{\bar \al-\al}+\frac{1}{M})|x|,
\]
which is equivalent to
\[
\Lda N^+(x)-\lda N^-(x) \ge -C(\rho^{\bar \al-\al}+\frac{1}{M}) |x|.
\]
The same computation with $\mathcal M^-_2(\tilde u_x-\tilde u)(0)\le C(\rho^{\bar \al-\al}+\frac{1}{M}) |x|$ provides the other inequality.
\end{proof}

One may consider $\bar w_A=(C(1/M+\rho^{\bar\al-\al}))^{-1}w_A(rx)$, where $C$ is the constant in \eqref{eq:bound wa}.  For every $\va_1$ small, we can choose $r$ smaller so that
\begin{align}
&\mbox{for evry set } A:~ |w_A| \le 1 \quad \mbox{in }\R^n, \nonumber\\&
\mbox{for evry set } A:~ \mathcal M^+_2 w_A \ge -\va_1 \quad \mbox{in }B_{1}\nonumber\\&
\frac{\lda}{\Lda} N^-(x)-\va_1|x| \le N^+(x)\le \frac{\Lda}{\lda} N^-(x)+\va_1|x|. \label{eq: holder key esti}
\end{align}
Note that $w_A$ and $\bar w_A$ share the same H\"older exponent.
\begin{lem}\label{lem:sigma+al}
We have for $x\in B_{1/4}$,
\[
N^+(x)\le C(1/M+\rho^{\bar\al-\al})|x|^{\bar\al}.
\]
\end{lem}
\begin{proof}
It follows from exactly the same proof of Lemma 9.2 in \cite{CS11}.
\end{proof}

\begin{proof}[Proof of Theorem \ref{thm:E-K-3}]
For $x\in B_{1/4}$, we have
\[
\begin{split}
& |-\Delta)^{\sigma/2}\tilde v_{i+1}(x)-(-\Delta)^{\sigma/2}\tilde v_{i+1}(0)|\\
&\quad =C\left|N^+(x)-N^-(x)+\int_{\R^n \setminus B_{1}} (\delta \tilde v_{i+1}(x,y)-\delta \tilde v_{i+1}(0,y))K(y)\,\ud y\right|\\
&\quad\le C(\rho^{\bar \al-\al}+\frac{1}{M}) |x|^{\bar\al}+C(\rho^{\bar \al-\al}+\frac{1}{M}) |x|\\&
\quad\le C(\rho^{\bar \al-\al}+\frac{1}{M})|x|^{\bar\al},
\end{split}
\]
where in the first inequality we used Lemma \ref{lem:sigma+al}, Lemma \ref{lem:9.1} and \eqref{eq:error delta}.

On the other hand, it follows from the computations in \eqref{eq:r1 holder} and Lemma \ref{lem:derivative-2} that
\[
|(-\Delta)^{\sigma/2} R^{(1)}_\rho(x)-(-\Delta)^{\sigma/2} R^{(1)}_\rho(0)|\le  C(\rho^{\bar \al-\al}+\frac{1}{M}) |x|^{\bar\al}.
\]
Thus,
\[
|(-\Delta)^{\sigma/2} v_{i+1}(x)-(-\Delta)^{\sigma/2} v_{i+1}(0)|\le C(\rho^{\bar \al-\al}+\frac{1}{M}) |x|^{\bar\al}.
\]
It follows from Lemma \ref{lem:7.2}, standard translation arguments and Schauder estimates for $(-\Delta)^{\sigma/2}$ that
\[
\|v_{i+1}\|_{C^{\sigma+\bar\al}(B_{1})}\le C(\rho^{\bar \al-\al}+\frac{1}{M}).
\]
This finishes the proof of Theorem \ref{thm:E-K-3} provided that $\rho_0^{\bar \al-\al}\le 1/(2C)$ and $M\ge 2C$.
\end{proof}

Lastly, let us discuss the case $0<\sigma_0\le\sigma\le 1$. In this case, the Evans-Krylov theorem in \cite{CS11} does not provide any improvement with respect to the $C^{1,\al}$ estimate in \cite{CS09}. However, we do not know how to use the incremental quotients method as in \cite{CS09} to prove our Theorem \ref{thm:E-K-3}.  But we still can find some $\bar\al>0$ so that Theorem \ref{thm:E-K-3} holds. Recall that in the proof of Theorem \ref{thm:E-K-3} above, there are two places where we used $\sigma>1$:
\begin{itemize}
\item [(i):] In \eqref{eq:estimate of T}, we used $\sigma\ge\sigma_0>1>\bar\al$ so that the integral there is universally bounded;

\item [(ii):] In \eqref{eq:grad est u}, we have the gradient estimate for $v_{i+1}$ when $\sigma\ge\sigma_0>1$. This was used in proving \eqref{eq:error delta} in the proof of Lemma \ref{lem:9.1} and \eqref{eq: holder key esti}.
\end{itemize}
It is clear that the use in (i) is not essential, since we can assume that $\bar\al < \sigma_0$ when $0<\sigma_0\le\sigma\le 1$. The use in (ii) is not essential, either,  since we can proceed using the H\"older estimates in \cite{CS09} that
\be\label{eq:holder vi+1}
\|v_{i+1}\|_{C^{\beta}(B_{1/2})}\le C(\rho^{\bar \al-\al}+\frac{1}{M})
\ee
instead of \eqref{eq:grad est u}, where $\beta\in (0,1)$ is a constant depending only on $n,\sigma_0,\lda,\Lda$. Consequently, the statement of  Lemma \ref{lem:9.1} becomes
\[
\frac{\lda}{\Lda} N^-(x)-C(\rho^{\bar \al-\al}+1/M) |x|^\beta \le N^+(x) \le \frac{\Lda}{\lda} N^-(x)+C(\rho^{\bar \al-\al}+1/M) |x|^\beta\quad\forall \ x\in B_{1/4},
\]
and \eqref{eq: holder key esti} becomes
\[
\frac{\lda}{\Lda} N^-(x)-\va_1|x|^{\beta} \le N^+(x)\le \frac{\Lda}{\lda} N^-(x)+\va_1|x|^{\beta}.
\]
The same proof of Lemma 9.2 in \cite{CS11} will give that there exists some $\bar\beta>0$ depending only on $\sigma_0,n,\lda,\Lda$ such that
\[
N^+(x)\le C(1/M+\rho^{\bar\al-\al})|x|^{\bar\beta}\quad\forall \ x\in B_{1/4},
\]
and we will choose $\bar\al=\bar\beta$ (which might be smaller than the one in \eqref{eq:E-K-1} when $\sigma<1$ if one consider the best possible one due to the $C^{1,\al}$ estimates in \cite{CS09} even for $\sigma$ very small). 

Thus, we can prove that 
\begin{thm}\label{thm:E-K-3-1}
For $\sigma_0\in (0,2)$ and $\sigma\in [\sigma_0,2)$, there exists a constant $\bar\al\in (0,1)$ depending only on $n,\sigma_0,\lda $ and $ \Lda$ so that Theorem \ref{thm:E-K-3} holds.
\end{thm}

\section{Schauder estimates}\label{sec:schauder}
In this section, we will prove the Schauder estimates in Theorem \ref{thm:schauder-nl}. We start with a lemma. It follows quickly from comparison principles and we omit the proof here.
\begin{lem}\label{lem:mp}
Suppose that every $K_a(y)\in\mathscr L_2(\lda,\Lda,\sigma)$ with $\sigma\ge\sigma_0>0$, $c_0$ is a constant. Let $u$ be the viscosity solution of
\[
\begin{split}
\inf_{a\in\mathcal A}\int_{\R^n} \delta u(x,y)K_a(y)\,\ud y&=c_0\quad\mbox{in }B_1\\
u&=g\quad\mbox{in }\R^n\setminus B_1.
\end{split}
\]
Then there exists a constant $C$ depending only on $\lda$, $\Lda$, $n$ and $\sigma_0$ such that
\[
 \|u\|_{L^\infty(\R^n)}\le C( \|g\|_{L^\infty(\R^n\setminus B_1)}+ |c_0|).
\]
\end{lem}

\begin{proof}[Proof of Theorem \ref{thm:schauder-nl}]
The strategy of the proof is to find a sequence of approximation solutions which are sufficiently regular, and the error between the genuine solution and the approximation solutions can be controlled in a desired rate. We divide the proof into four steps.

\medskip

\emph{Step 1}: Normalization and rescaling.

\medskip
Let $w_0$ be the viscosity solution of
\[
\begin{split}
I_0 w_0(x):=\inf_{a\in\mathcal A}\int_{\R^n} \delta w_0(x,y)K_a(0,y)\,\ud y-f(0)&=0\quad\mbox{in }B_4\\
w_0&=u\quad\mbox{in }\R^n\setminus B_4.
\end{split}
\]
Then by Lemma \ref{lem:mp} we have that
\[
 \|w_0\|_{L^\infty(\R^n)}\le C( \|u\|_{L^\infty(\R^n)}+ \|f\|_{L^\infty(B_5)}).
\]
Thus by normalization, we may assume that
\[
  \|w_0\|_{L^\infty(\R^n)}\le 1/2, \ \|u\|_{L^\infty(\R^n)}+ \|f\|_{L^\infty(B_5)}\le 1/2.
\]

For some universal small positive constant $\gamma<1$, which will be chosen later in \eqref{eq:choice of eta}, we may also assume that $|f(x)-f(0)|\le\gamma |x|^\al$ and
\be\label{eq:K holder small}
 \int_{\R^n}|K_a(x,y)-K_a(0,y)|\min(|y|^2,r^2)dy \le \gamma |x|^\al r^{2-\sigma}
\ee
for all $a\in\mathcal A, r\in (0,1]$, $x\in B_5$. This can be achieved by the scaling
for $s<1$ small that if we let
\be\label{eq:rescaling procedure}
\begin{split}
\tilde K_a(x,y)&= s^{n+\sigma}K_a(sx,sy)\in\mathscr L_2(\lda,\Lda,\sigma),\\
 \tilde u(x)&=u(sx), \\
 \tilde f(x)&=s^{\sigma}f(sx),
\end{split}
\ee
then we see that
\[
\tilde I \tilde u(x)=\inf_{a\in\mathcal A}\tilde L_a\tilde u(x)=\tilde f(x)\quad \mbox{in }B_5,
\]
where
\[
\tilde L_a \tilde u(x):=\int_{\R^n} \delta \tilde u(x,y) \tilde K_a(x,y)\,\ud y.
\]
It follows that if we choose $s$ sufficiently small, then
\[
|\tilde f(x)-\tilde f(0)|\leq M_f s^{\sigma+\al}|x|^\al\leq  \gamma |x|^\al\le 5\gamma,
\]
and
\[
 \int_{\R^n}|\tilde K_a(x,y)-\tilde K_a(0,y)|\min(|y|^2,r^2)\ud y \le 2\Lda s^\al |x|^\al r^{2-\sigma}\le\gamma |x|^\al r^{2-\sigma}
\]
for all $a\in\mathcal A, r\in (0,1]$, $x\in B_5$. Thus, we may consider the equation of $\tilde u$ instead.

Consequently, it follows from \eqref{eq:K holder small} that ($\|\cdot\|_*$ is defined in \eqref{eq:norm star} in the Appendix)
\[
 \| I- I_0\|_*\le 25 \gamma.
\]
Indeed, if $x\in B_5, h\in C^2(x), \|h\|_{L^\infty(\R^n)}\le M, |h(y)-h(x)-(y-x)\cdot \nabla h(x)|\le \frac{M}{2}|x-y|^2$  for every $y\in B_1(x)$, we have
\begin{align}
\| I- I_0\|_* &\leq \sup_{x,a,h}\frac{1 }{1+M} \int_{\R^n} |\delta h(x,y)|| K_a (x,y)- K_a(0,y)|\,\ud y\nonumber\\&
\le \sup_{a}\frac{M}{1+M} \left( \int_{B_1}  |y|^2 | K_a (x,y)- K_a(0,y)| +4\int_{\R^n\setminus B_1}   | K_a (x,y)- K_a(0,y)|\right)\nonumber\\&
<5\gamma|x|^\al\le 25\gamma.\label{eq:small operator-1}
\end{align}

\medskip

\emph{Step 2}: From now on, we denote
\[
\rho=\rho_0\mbox{ as the one in Theorem \ref{thm:E-K-3}, which is a universal constant}.
\]
We claim that we can find a sequence of functions $w_i$, $i=0,1,2,\cdots$, such that for all $i$,
\be\label{eq:for k-nl}
\inf_{a\in\mathcal A}\int_{\R^n} \sum_{\ell=0}^{i}\delta w_\ell(x,y)K_a(0,y)\ud y=f(0)\quad \mbox{in }B_{4\cdot \rho^{i}},
\ee
and
\be\label{eq:zero outside-nl}
(u-\sum_{l=0}^i w_\ell)(\rho^{i}x)=0\quad\mbox{for all }x\in \R^n\setminus B_4,
\ee
and
\be\label{eq:est for k-nl}
\begin{split}
\|w_i\|_{L^\infty(\R^n)}&\le \rho^{(\sigma+\al)i},\\
\|D w_i\|_{L^\infty(B_{(4-\tau)\cdot \rho^{i}})}&\le c_2 \rho^{(\sigma+\al-1)i}\tau^{-1},\\
\|D^2 w_i\|_{L^\infty(B_{(4-\tau)\cdot \rho^{i}})}&\le c_2 \rho^{(\sigma+\al-2)i}\tau^{-2},\\
[D^2 w_i]_{C^{\sigma+\bar\al-2}(B_{(4-\tau)\cdot \rho^{i}})}&\le c_2 \rho^{(\al-\bar\al)i}\tau^{-4},\\
\end{split}
\ee
and
\be\label{eq:close app-nl}
\|u-\sum_{\ell=0}^i w_\ell\|_{L^\infty(\R^n)}\le \rho^{(\sigma+\al)(i+1)},
\ee
and
\be\label{eq:holer close app-nl}
[ u-\sum_{\ell=0}^i w_\ell]_{C^{\al_1}( B_{(4-3\tau)\cdot \rho^{i}})}\le 8c_1 \rho^{(\sigma+\al-\al_1)i}\tau^{-4},
\ee
where $\tau$ is an arbitrary constant in $(0,1]$, $\al_1$ and $c_1$ are positive constants depending only on $n$, $\lda$, $\Lda$, $\gamma_0$ and $\bar\al$, and $c_2$ is the constant in \eqref{eq:E-K-3}.

Then Theorem \ref{thm:schauder-nl} will follow from this claim and standard arguments. Indeed, we have,  when $1<\sigma+\al<2$ and for $\rho^{i+1}\le |x|<\rho^{i}$,
\[
\begin{split}
&|u(x,0)-\sum_{\ell=0}^\infty w_\ell(0,0)-\sum_{\ell=0}^\infty \nabla_x w_\ell(0,0)\cdot x|\\
&\le |u(x,0)-\sum_{\ell=0}^i w_\ell(x,0)|+|\sum_{\ell=0}^i w_\ell(x,0)-\sum_{\ell=0}^i w_\ell(0,0)-\sum_{\ell=0}^i \nabla_x w_\ell(0,0)\cdot x|\\
&\quad+|\sum_{\ell=i+1}^\infty w_\ell(0,0)|+|\sum_{\ell=i+1}^\infty \nabla_x w_\ell(0,0)\cdot x|\\
&\le \rho^{(\sigma+\al)(i+1)}+c_2 |x|^2\sum_{\ell=0}^i\rho^{(\sigma+\al-2)\ell}+\sum_{\ell=i+1}^\infty  \rho^{(\sigma+\al)\ell}+|x|\sum_{\ell=i+1}^\infty c_2 \rho^{(\sigma+\al-1)\ell}\\
&\le C_2|x|^{\sigma+\al}.
\end{split}
\]

When $\sigma+\al>2$ and for $\rho^{i+1}\le |x|<\rho^{i}$,
\[
\begin{split}
&|u(x)-\sum_{\ell=0}^\infty w_\ell(0)-\sum_{\ell=0}^\infty Dw_\ell(0)\cdot x-\sum_{\ell=0}^\infty \frac 12 x^TD^2w_\ell(0)x|\\
&\le |u(x)-\sum_{\ell=0}^i w_\ell(x)|+|\sum_{\ell=0}^i w_\ell(x)-\sum_{\ell=0}^i w_\ell(0)-\sum_{\ell=0}^i Dw_\ell(0)\cdot x-\sum_{\ell=0}^i \frac 12 x^TD^2w_\ell(0)x|\\
&\quad+|\sum_{\ell=i+1}^\infty w_\ell(0)|+|\sum_{l=i+1}^\infty Dw_\ell(0)\cdot x|+\frac 12 |\sum_{\ell=i+1}^\infty x^T D^2w_\ell(0)x|\\
&\le \rho^{(\sigma+\al)(i+1)}+2 c_2 |x|^{\sigma+\bar\al}\sum_{\ell=0}^i\rho^{(\al-\bar\al)\ell}+\sum_{\ell=i+1}^\infty  \rho^{(\sigma+\al)\ell}+|x|\sum_{\ell=i+1}^\infty c_2 \rho^{(\sigma+\al-1)\ell}\\
&\quad+|x|^2\sum_{\ell=i+1}^\infty c_2 \rho^{(\sigma+\al-2)\ell}\\
&\le C_3|x|^{\sigma+\al}.
\end{split}
\]
This proves the estimate \eqref{eq:schauder estimate}.

Now we are left to prove this claim. Before we provide the detailed proof, we would like to first mention the idea and the structure of \eqref{eq:for k-nl}-\eqref{eq:holer close app-nl}:
\begin{itemize}
\item Solving \eqref{eq:for k-nl} and \eqref{eq:zero outside-nl} inductively is how we construct this sequence of functions $\{w_i\}$.

\item \eqref{eq:close app-nl} will follow from the approximation lemmas in the appendix, where \eqref{eq:holer close app-nl} will be used. 

\item \eqref{eq:est for k-nl} will follow from \eqref{eq:close app-nl}, maximum principles and the recursive Evans-Krylov theorem, Theorem \ref{thm:E-K-3}.

\end{itemize}

\medskip

\emph{Step 3}: Prove the claim for $i=0$.

\medskip

Let $u$ be a viscosity solution of \eqref{eq:main schauder}.
It follows from the H\"older estimates in \cite{CS09}, standard scaling and covering (contributing at most a factor of $4/\tau$) arguments that there exist constants $\al_1\in (0,1), c_1>0$, depending only on $n,\lda, \Lda, \gamma_0, \bar\al$,  such that for $\tau\in (0,1]$
\be \label{eq:holder}
\|u\|_{C^{\al_1}(B_{4-\tau})} \leq c_1\tau^{-1-\al_1}\left(\|u\|_{L^\infty(\R^n)}+\|f\|_{L^\infty(B_4)}\right).
\ee

Let $w_0$ be the one in Step 1 and $c_2$ be the constant in \eqref{eq:E-K-3}. Then by Theorem \ref{thm:E-K-3}, standard scaling, translation and covering arguments that
\be\label{eq:w0 esti}
\begin{split}
\|w_0\|_{L^\infty(\R^n)}\leq 1, \quad &\|D w_0\|_{L^\infty(B_{4-\tau})} \leq c_2 \tau^{-1}, \\
\quad \|D^2 w_0\|_{L^\infty(B_{4-\tau})} \leq c _2 \tau^{-2},\quad &[D^2 w_0]_{C^{\sigma+\bar\al-2}(B_{4-\tau})} \leq c _2\tau^{-4}.
\end{split}
\ee

Let us set up to apply the approximation lemma, Lemma \ref{lem:appr0}, in the Appendix. Let $\va=\rho^{3}\le\rho^{\sigma+\al}$ and $M=1$. Let us fixed a modulus continuity $\omega_1(r)=r^{\al_1}$. Then for these $\omega_1,\va, M$, there exist $\eta_1$ (small) and $R$ (large) so that Lemma \ref{lem:appr0} holds. We can assume that the rescaling in \eqref{eq:rescaling procedure} make the equation  hold in a very large ball containing $B_{2R}$ and $|u(x)-u(y)|\le\omega_1(|x-y|)$ for every $x\in B_R\setminus B_4$ and $y\in\R^n\setminus B_4$. The latter one can be done due to \eqref{eq:holder}. We will choose $\gamma<\eta_1/25$ in \eqref{eq:choice of eta}.
Then by the rescaling in Step 1, we can conclude from Lemma \ref{lem:appr0} that
\[
\|u-w_0\|_{L^\infty(B_4)}\leq \va\le \rho^{\sigma+\al},
\]
and thus,
\[
\|u-w_0\|_{L^\infty(\R^n)}\le \|u-w_0\|_{L^\infty(B_4)}\leq \va\le \rho^{\sigma+\al}.
\]
This proves that \eqref{eq:for k-nl}, \eqref{eq:zero outside-nl}, \eqref{eq:est for k-nl} and \eqref{eq:close app-nl} hold for $i=0$.

Let $
v(x)=u(x)-w_{0}(x).
$
Since $w_0\in C^{\sigma+\bar\al}$, $v$ is a solution of
\[
\begin{split}
I^{(0)} v:&=\inf_{a\in \mathcal A}\int_{\R^n}\delta v(x,y)K_a (x,y)+\delta w_{0}(x,y)K_a(x,y)\ud y-  f(0)\\
&= f( x)-  f(0)\quad\mbox{in }B_{4}.
\end{split}
\]
It is clear that $I^{(0)}$ is elliptic with respect to $\mathscr L_0(\lda,\Lda,\sigma)$. Moreover, for $x\in B_{4-2\tau}$,
\begin{align}
|I^{(0)} 0|:&=|\inf_{a\in \mathcal A}\int_{\R^n}\delta w_{0}(x,y)K_a(x,y)\ud y- f(0)|\nonumber\\
&=|\inf_{a\in \mathcal A}\int_{\R^n}\delta (w_0(x,y))K_a(x,y)\ud y- \inf_{a\in \mathcal A}\int_{\R^n}\delta (w_0(x,y))K_a(0,y)\ud y|\nonumber\\
&\le \sup_{a\in \mathcal A} \int_{\R^n}|\delta w_0( x,y)||K_a( x,y)-K_a(0,y)|\ud y\nonumber\\
&\le  \sup_{a\in \mathcal A}\left(\int_{B_\tau}c_2\tau^{-2}|y|^2|K_a(x,y)-K_a(0,y)|\ud y +4\int_{\R^n\setminus B_\tau}|K_a( x,y)-K_a(0,y)|\ud y\right)\nonumber\\
&\le \gamma\left(c_2+4\right)|x|^\al\tau^{-\sigma}\le \gamma4\left(c_2+4\right)\tau^{-\sigma}\le\tau^{-\sigma},\label{eq:cal error}
\end{align}
where \eqref{eq:w0 esti} was used in the second inequality, and \eqref{eq:K holder small} was used in the third inequality, and \eqref{eq:choice of eta} was used in the last inequality. It follows from H\"older estimates established in \cite{CS09}, standard scaling and covering arguments (contributing at most a factor of $4/\tau$) we have
\[
\|v\|_{C^{\al_1}(B_{4-3\tau})}\le c_1\tau^{-\al_1-1}(\tau^{-\sigma}+4\gamma +1)\le 8c_1\tau^{-4},
\]
and thus,
\[
[u-w_0]_{C^{\al_1}(B_{4-3\tau})}\le 8c_1\tau ^{-4}.
\]
This finishes the proof of \eqref{eq:holer close app-nl} for $i=0$. 

\medskip

\emph{Step 4}: We assume all of \eqref{eq:for k-nl}, \eqref{eq:zero outside-nl}, \eqref{eq:est for k-nl}, \eqref{eq:close app-nl} and \eqref{eq:holer close app-nl} hold up to $i\ge 0$, and we will show that they all hold for $i+1$ as well. 

\medskip

Let
\[
\begin{split}
W(x)&=\rho^{-(i+1)(\sigma+\al)}\left(u-\sum_{\ell=0}^i w_\ell\right)(\rho^{i+1}x),\\
v_\ell&=\rho^{-(\sigma+\al)\ell}w_{\ell}(\rho^\ell x),
\end{split}
\]
and
\[
K^{(i+1)}(x,y)=\rho^{(n+\sigma)(i+1)}K(\rho^{i+1}x,\rho^{i+1}y).
\]
Since $w_\ell\in C^{\sigma+\bar\al}$ for each $\ell$, then $W$ is a solution of
\[
\begin{split}
I^{(i+1)}W=\rho^{-(i+1)\al}f(\rho^{i+1}x)-\rho^{-(i+1)\al}f(0)\quad\mbox{in } B_{4/\rho},
\end{split}
\]
where
\[
\begin{split}
&I^{(i+1)}W\\
&:= \inf_{a\in\mathcal A}\int_{\R^n} \left(\delta W(x,y)+\sum_{\ell=0}^{i}\rho^{-(i+1)(\sigma+\al)}\delta w_\ell(\rho^{i+1}x,\rho^{i+1}y)\right)K^{(i+1)}_a(x,y)\ud y\\
&\quad\quad\quad-\rho^{-(i+1)\al}f(0)\\
&= \inf_{a\in\mathcal A}\int_{\R^n} \left(\delta W(x,y)+\sum_{\ell=0}^{i}\rho^{-(i+1-\ell)(\sigma+\al)}\delta v_\ell(\rho^{i+1-\ell}x,\rho^{i+1-\ell}y)\right)K^{(i+1)}_a(x,y)\ud y\\
&\quad\quad\quad-\rho^{-(i+1)\al}f(0).
\end{split}
\]
It is clear that $I^{(i+1)}$ is elliptic with respect to $\mathscr L_0(\lda,\Lda,\sigma)$. Denote
\[
\begin{split}
&I^{(i+1)}_0v\\
&:=\inf_{a\in\mathcal A}\int_{\R^n} \left(\delta v(x,y)+\sum_{\ell=0}^{i}\rho^{-(i+1)(\sigma+\al)}\delta w_\ell(\rho^{i+1}x,\rho^{i+1}y)\right)K^{(i+1)}_a(0,y)\ud y\\
&\quad\quad\quad-\rho^{-(i+1)\al}f(0)\\
&=\inf_{a\in\mathcal A}\int_{\R^n} \left(\delta v(x,y)+\sum_{\ell=0}^{i}\rho^{-(i+1-\ell)(\sigma+\al)}\delta v_\ell(\rho^{i+1-\ell}x,\rho^{i+1-\ell}y)\right)K^{(i+1)}_a(0,y)\ud y\\
&\quad\quad\quad-\rho^{-(i+1)\al}f(0),
\end{split}
\]
which is also elliptic with respect to $\mathscr L_0(\lda,\Lda,\sigma)$.
Let $v_{i+1}$ be the solution of
\[
\begin{split}
I^{(i+1)}_0v_{i+1}&=0\quad\mbox{in }B_4\\
v_{i+1}&=W\quad\mbox{in }\R^n\setminus B_4.
\end{split}
\]
It follows that
\be\label{eq:uniform bound for sequence}
\|v_{i+1}\|_{L^\infty(\R^n)}\le \|W\|_{L^\infty(\R^n)}\le 1.
\ee
Indeed, we first know from the nonlocal Evans-Krylov theorem that $v_{i+1}\in C^{\sigma+\bar\al}$ and thus $I^{(i+1)}_0v_{i+1}$ can be calculated point-wisely. Since $I^{(i+1)}_0 0=0$ which follows from \eqref{eq:for k-nl}, we have for $x\in B_4$,
\[
 \inf_{a\in\mathcal A}\int_{\R^n}\delta v_{i+1}(x,y)K^{(i+1)}_a(0,y)\ud y\le I^{(i+1)}_0v_{i+1}(x)\le \sup_{a\in\mathcal A}\int_{\R^n}\delta v_{i+1}(x,y)K^{(i+1)}_a(0,y)\ud y.
\]
 We also know from then boundary regularity in \cite{CS10} that $v_{i+1}\in C(\overline B_4)$. Suppose that there exists $x_0\in B_4$ so that $v_{i+1}(x_0)=\max_{\overline B_4}v_{i+1} >\|W\|_{L^\infty(\R^n\setminus B_4)}$. Then
\[
 \sup_{a\in\mathcal A}\int_{\R^n}\delta v_{i+1}(x_0,y)K^{(i+1)}_a(0,y)\ud y<0,
\]
which is a contradiction to $I^{(i+1)}_0v_{i+1}(x_0)=0$. Similarly, we have $v_{i+1}(x)\ge -\|W\|_{L^\infty(\R^n\setminus B_4)}$ for $x\in B_4$. This proves \eqref{eq:uniform bound for sequence}.

Again, by our induction hypothesis \eqref{eq:for k-nl}, it follows that for all $m=0,1,\cdots,i,$
\[
\inf_{a\in\mathcal A}\int_{\R^n} \left(\sum_{\ell=0}^{m}\rho^{-(m-\ell)(\sigma+\al)}\delta v_\ell(\rho^{m-\ell}x,\rho^{m-\ell}y)\right)K^{(m)}_a(0,y)\ud y=\rho^{-m\al}f(0)\quad\mbox{in }B_{4}.
\]
It follows from Theorem \ref{thm:E-K-3} and standard scaling arguments that 
\[
\begin{split}
\|D v_{i+1}\|_{L^\infty(B_{4-\tau})}&\le c_2\tau^{-1},\\
 \|D^2 v_{i+1}\|_{L^\infty(B_{4-\tau})}&\le c_2 \tau^{-2},\\
 [D^2 v_{i+1}]_{C^{\sigma+\bar\al-2}(B_{4-\tau})}&\le c_2 \tau^{-4}.
\end{split}
\]

We want to apply Lemma \ref{lem:appr} to the equations of $W$ and $v_{i+1}$ so that we have $|W-v_{i+1}|\le \rho^{\sigma+\al}$ in $B_4$. 

First of all, $|W|\le 1$ in $\R^n$, $W\equiv 0$ in $\R^n\setminus B_{4/\rho}$, and $[W]_{C^{\al_1}(B_{(4-3\tau)/\rho})}\le 8c_1\rho^{\al_1-\sigma-\al}\tau^{-4}\le 8c_1\rho^{-3}\tau^{-4}$. Secondly, it follows from similar computations in \eqref{eq:r1 holder}, and making use of \eqref{eq:est for k-nl} and Lemma \ref{lem:derivative-3} that
\[
[L^{(i+1)}_a R_\rho]_{C^{\bar\al}(B_4)}\le M_0\quad\forall~a\in\mathcal A,
\]
where $M_0$ is a universal constant independent of $i$,
\[
\begin{split}
L^{(i+1)}_a v=\int_{\R^n}\delta v(x,y)K^{(i+1)}_a(0,y)\,\ud y,\quad
R_\rho(x)=\sum_{\ell=0}^i \rho^{-(i+1-\ell)(\sigma+\al)} v_\ell(\rho^{i+1-\ell}x).
\end{split}
\]
Lastly, we are going to show that we can choose $\gamma$ sufficiently small so that
\be\label{eq:small norm i+1}
\|I^{(i+1)}-I_0^{(i+1)}\|_*\le\eta_2\quad\mbox{in }B_4
\ee
and we can apply Lemma \ref{lem:appr}, where $\eta_2$ is the one in \eqref{lem:appr} with $\va=\rho^{3}\le\rho^{\sigma+\al}$, $M_0$ as above, $M_1=1, M_2=8c_1\rho^{-3}$, $M_3=c_2$. 

For $x\in B_4, h\in C^2(x), \|h\|_{L^\infty(\R^n)}\le M, |h(y)-h(x)-(y-x)\cdot \nabla h(x)|\le \frac{M}{2}|x-y|^2$  for every $y\in B_1(x)$, we have
\[
\begin{split}
&\|I^{(i+1)}-I_0^{(i+1)}\|_*\\
&\quad\le \sup_{a, h, x}|\int_{\R^n}\delta h(x,y)(K^{(i+1)}_a(x,y)-K^{(i+1)}_a(0,y))\ud y|\\
&\quad\quad+\sum_{\ell=0}^{i}\sup_{a\in\mathcal A}|\int_{\R^n}\rho^{-(i+1)(\sigma+\al)}\delta w_\ell(\rho^{i+1}x,\rho^{i+1}y)(K^{(i+1)}_a(x,y)-K^{(i+1)}_a(0,y))\ud y|\\
&\quad=I_1+I_2.\\
\end{split}
\]
It follows from the same computations in \eqref{eq:small operator-1} that
\[
|I_1|\le25\gamma.
\]
For $a\in\mathcal A$, $\ell=0,1,\cdots, i$ and for $x\in B_{(4-2\tau)/\rho}$, we have, similar to \eqref{eq:cal error},
\be\label{eq:acc error}
\begin{split}
&|\int_{\R^n} \delta w_\ell (\rho^{i+1}x,\rho^{i+1}y)(K_a^{(i+1)}(0,y)-K_a^{(i+1)}(x,y))\ud y|\\
&\le\rho^{\sigma(i+1)}\int_{\R^n}|\delta w_\ell(\rho^{i+1}x,y)||K_a(0,y)-K_a(\rho^{i+1}x,y)|\ud y\\
&\le \rho^{\sigma(i+1)}\int_{B_{\rho^{\ell}\tau}}c_2\rho^{(\sigma+\al-2)\ell}\tau^{-2}|y|^2|K_a(0,y)-K_a(\rho^{i+1}x,y)|\ud y\\
&\quad\quad +\rho^{\sigma(i+1)}\int_{\R^n\setminus B_{\rho^{\ell}\tau}}\rho^{(\sigma+\al)\ell}4|K_a(0,y)-K_a(\rho^{i+1}x,y)|\ud y\\
&\le \rho^{(\sigma+\al)(i+1)}\gamma(c_2+4)\rho^{\al \ell}\tau^{-\sigma}|x|^\al,
\end{split}
\ee
where we used \eqref{eq:est for k-nl} in the second inequality. We choose $\gamma$ such that 
\be\label{eq:choice of eta}
\left(25+(c_2+4)4\sum_{\ell=0}^\infty\rho^{\al \ell}\right)\gamma\le\min(\eta_1/25, \eta_2).
\ee
It follows that \eqref{eq:small norm i+1} holds (here we can choose $\tau=1$). By Lemma \ref{lem:appr}  we have that
\[
\|W-v_{i+1}\|_{L^\infty(\R^n)}=\|W-v_{i+1}\|_{L^\infty(B_4)}\le\va\le\rho^{\sigma+\al}.
\]
Let
\[
w_{i+1}(x)=\rho^{(\sigma+\al)(i+1)}v_{i+1}(\rho^{-(i+1)}x).
\]
Thus, we have shown in the above that all of \eqref{eq:for k-nl}, \eqref{eq:zero outside-nl}, \eqref{eq:est for k-nl}, \eqref{eq:close app-nl} hold for $i+1$. In the following, we shall show that \eqref{eq:holer close app-nl} hold for $i+1$ as well. Let
\[
V=W-v_{i+1}=\rho^{-(i+1)(\sigma+\al)}\left(u-\sum_{\ell=0}^{i+1} w_\ell\right)(\rho^{i+1}x).
\]
Thus, for $x\in B_{4}$
\[
\begin{split}
&I^{(i+1)}V\\
:&=\inf_{a\in\mathcal A}\int_{\R^n} [\delta V(x,y)+\sum_{\ell=0}^{i+1}\rho^{-(i+1)(\sigma+\al)}\delta w_\ell(\rho^{i+1}x,\rho^{i+1}y)]K^{(i+1)}_a(x,y)\ud y-\rho^{-(i+1)\al}f(0)\\
&=\rho^{-(i+1)\al}f(\rho^{i+1}x)-\rho^{-(i+1)\al}f(0).
\end{split}
\]
Moreover, for $x\in B_{4-2\tau}$,
\[
\begin{split}
|I^{(i+1)}0|&=|\inf_{a\in\mathcal A}\int_{\R^n} [\sum_{\ell=0}^{i+1}\rho^{-(i+1)(\sigma+\al)}\delta w_\ell(\rho^{i+1}x,\rho^{i+1}y)]K^{(i+1)}_a(x,y)\ud y-\rho^{-(i+1)\al}f(0)|\\
&=|\inf_{a\in\mathcal A}\int_{\R^n} [\sum_{\ell=0}^{i+1}\rho^{-(i+1)(\sigma+\al)}\delta w_\ell(\rho^{i+1}x,\rho^{i+1}y)]K^{(i+1)}_a(x,y)\ud y\\
&\quad-\inf_{a\in\mathcal A}\int_{\R^n} [\sum_{\ell=0}^{i+1}\rho^{-(i+1)(\sigma+\al)}\delta w_\ell(\rho^{i+1}x,\rho^{i+1}y)]K^{(i+1)}_a(0,y)\ud y|\\
&\le\sup_{a\in\mathcal A}\sum_{l=0}^{i+1}\int_{\R^n} \rho^{-(i+1)(\sigma+\al)}|\delta w_\ell(\rho^{i+1}x,\rho^{i+1}y)||K^{(i+1)}_a(x,y)-K^{(i+1)}_a(0,y)|\ud y\\
&\le \eta_2 \tau^{-\sigma},
\end{split}
\]
where in the last inequality we have used \eqref{eq:acc error} and the choice of $\eta_2$ in \eqref{eq:choice of eta}. Thus, by standard scaling and covering arguments,
\[
[V]_{C^{\al_1}(B_{4-3\tau})}\le 8c_1\tau^{-4}.
\]
Hence, \eqref{eq:holer close app-nl} holds for $i+1$.

This finishes the proof of the claim in Step 2. Therefore, the proof of Theorem \ref{thm:schauder-nl} is completed.
\end{proof}

\begin{rem}\label{rem:how to appr}
In the step of approximation, one cannot use
\[
\begin{split}
&\tilde I^{(i+1)}_0v\\
&:=\inf_{a\in\mathcal A}\int_{\R^n} \left(\delta v(x,y)+\sum_{\ell=0}^{i}\rho^{-(i+1-\ell)(\sigma+\al)}\delta v_\ell(0,\rho^{i+1-\ell}y)\right)K^{(i+1)}_a(0,y)\ud y-\rho^{-(i+1)\al}f(0)
\end{split}
\]
to approximate $I^{(i+1)}W$, since one can check that $\tilde I^{(i+1)}_0$ will not be close to $I^{(i+1)}$. This is the main reason why we need Theorem \ref{thm:E-K-3}.
\end{rem}

\begin{rem}\label{rem:less than 2}
 In the case of $\sigma\ge \sigma_0>0$ and $\sigma+\bar\al\le2-\gamma_0$ for some $\gamma_0>0$, our approximation solutions $\{w_\ell\}$ are of only $C^{\sigma+\bar\al}$ but may not be $C^2$. Thus, instead of \eqref{eq:holder K}, we need the following (stronger) assumption on $K_a$:
\be\label{eq:holder K-2}
\int_{\R^n}|K_a(x,y)-K_a(0,y)|\min(|y|^{\sigma+\bar\al},r^{\sigma+\bar\al})dy \le \Lambda |x|^{\alpha} r^{\bar\al},
\ee
which will be used in \eqref{eq:cal error} and \eqref{eq:acc error}. Then, with the help of Theorem \ref{thm:E-K-3-1}, for $|\sigma+\bar\al-1|\ge\gamma_0 $, $\al\in (0,\bar\al)$ and $|\sigma+\al- 1|\ge \va_0$,  the same proof shows that the Schauder estimate \eqref{eq:schauder estimate} holds under the conditions \eqref{eq:holder K-2} and \eqref{eq:holder f}, where the  constant $C$ there will additionally depend on $\sigma_0$.
\end{rem}

Let $\sigma_0\in (0,2)$. A unified H\"older condition on the kernels $K$ for all $\sigma\in [\sigma_0,2)$, which is slightly stronger than both \eqref{eq:holder K} and \eqref{eq:holder K-2}, would be
\be\label{eq:holder K-3}
\int_{B_{2r}\setminus B_r}|K(x,y)-K(0,y)|dy\le (2-\sigma)\Lambda|x|^\al r^{-\sigma}
\ee
for all $r>0$, $x\in B_5$.

Combining Theorem \ref{thm:schauder-nl} and Remark \ref{rem:less than 2}, we have this corollary.
\begin{cor}\label{cor:unified}
Let $\sigma_0\in (0,2)$. There exists $\bar\al\in (0,1)$ depending only on $n,\lda,\Lda$ and $\sigma_0$ such that the following statement holds: Assume every $K_a(x,y)\in\mathscr L_2 (\lda,\Lda, \sigma)$ satisfies \eqref{eq:holder K-3} with $\sigma\in [\sigma_0,2)$, $\al\in(0,\bar\al)$, $|\sigma+\bar\al-j|\ge\gamma_0>0$ and $|\sigma+\al- j|\ge \va_0>0$ for $j=1,2$. Suppose that $f$ satisfies \eqref{eq:holder f}. If $u$ is a bounded viscosity solution of \eqref{eq:main schauder}, then there exists a polynomial $P(x)$ of degree $[\sigma+\al]$ such that \eqref{eq:schauder estimate} holds for $x\in B_1$, where $C$ in \eqref{eq:schauder estimate} is a positive constant depending only on $\lda,\Lda, n, \sigma_0, \bar\al, \al, \va_0$ and $\gamma_0$.

\end{cor}

An application of our Schauder estimates is another proof of the following Evans-Krylov type estimates for viscosity solutions of nonlocal fully nonlinear parabolic equations:
\be\label{eq:parabolic}
u_t(x,t)=\inf_{a\in\mathcal A} \left\{\int_{\R^n} \delta u(x,y;t)K_a(y)\,\ud y\right\}\quad\mbox{in }B_2\times (-2,0],
\ee
where $\delta u(x,y;t)=u(x+y,t)+u(x-y,t)-2u(x,t)$, $\mathcal A$ is an index set, and each $K_a\in\mathscr{L}_2(\lda, \Lda, \sigma)$. These estimates for more general nonlocal parabolic equations have been established by H. Chang Lara and G. Davila \cite{LD4}.  The definition of viscosity solutions to nonlocal parabolic equations and their many properties can be found in \cite{LD1,LD2}.  
\begin{thm}\label{thm:parabolic evans-krylov}
 Let $u: \R^n\times[-2,0]\to \R$ be a viscosity solution of \eqref{eq:parabolic}. Suppose that $u$ is Lipschitz continuous in $t$ in $(\R^n\setminus B_2)\times [-2,0]$ and $\|\mathcal M_{0}^{\pm} u(\cdot,-2)\|_{L^\infty(\R^n)}\le C_0$. Then there exists $\bar\beta\in (0,1)$ depending only on $n,\lda,\Lda$ such that for $\sigma+\bar\beta-2\ge\gamma_0>0$ we have
\be\label{eq:ek parabolic}
\begin{split}
 \|u_t\|_{C^{\bar\al}_{x,t}(B_1\times [-1,0])} & + \|\nabla_x^2 u\|_{C^{\bar\al}_{x,t}(B_{1}\times [-1,0])}\\
& \le C(\|u\|_{L^\infty(\R^n\times[-2,0])}+\|u_t\|_{L^\infty((\R^n\setminus B_2)\times [-2,0])}+C_0),
\end{split}
\ee
where $\bar\al=\gamma_0\bar\beta/2$ and $C$ is a positive constant depending only on $n,\lda,\Lda$ and $\gamma_0$.
\end{thm}
\begin{proof}
 It follows from Theorem 6.2 in \cite{LD1} and Theorem 4.1 in \cite{LD2} that there exists some $\bar\beta\in (0,1)$ depending only on $n,\lda,\Lda$ such that
\[
 \nabla_{x,t} u\in C^{\bar\beta}_{x,t}(B_1\times [-1,0]).
\]
In particular, the right hand side of \eqref{eq:parabolic} is H\"older in $x$. By the Schauder estimates in Theorem \ref{thm:schauder-nl} (and adjusting $\bar\beta$ if necessary), when $\sigma+\bar\beta-2\ge\gamma_0>0$, we have for all $t\in [-1,0]$
\[
 \nabla^2_x u(\cdot, t)\in C^{\sigma+\bar\beta-2}_x(B_1).
\]
By Lemma 3.1 on page 78 in \cite{LSU}, we have for all $x\in B_1$,
\[
  \nabla^2_x u(x, \cdot)\in C^{\bar\beta(\sigma+\bar\beta-2)/(\sigma+\bar\beta-1)}_t([-1,0])\subset C^{\bar\alpha}_t( [-1,0]).
\]
Thus, $\nabla^2_x u\in C^{\bar\alpha}_{x,t}(B_1\times [-1,0])$, and the estimate \eqref{eq:ek parabolic} follows from the estimates in Theorem 6.2 in \cite{LD1}, Theorem 4.1 in \cite{LD2} and the Schauder estimates we proved. This finishes the proof.
\end{proof}

Note that Example 2.4.1 in \cite{LD2} shows that the assumption of the Lipschitz continuity on $u$ in $(\R^n\setminus B_2)\times [-2,0]$ is necessary to obtain H\"older continuity of $u_t$ in $B_1\times [-1,0]$. The estimate \eqref{eq:ek parabolic} is not written in the scaling invariant form for the purpose of convenience in its proof. The constant $C$ in \eqref{eq:ek parabolic} does not depend on $\sigma$, and thus, does not blow up as $\sigma\to 2$.

One also can replace the condition on the initial data $u(\cdot,-2)$ in Theorem \ref{thm:parabolic evans-krylov} by the following global Lipschitz type assumption:
\be\label{eq:assuption global lip}
[u]_{C^{0,1}((t_1,t_2]; L^1(\omega_\sigma))}:=\sup_{(t-\tau,t]\subset(t_1,t_2]}\frac{\|u(\cdot, t)-u(\cdot,t-\tau)\|_{L^1(\omega_\sigma)}}{\tau}<\infty,
\ee
where
$
\|v\|_{L^1(\omega_\sigma)}=\int_{\R^n}|v(y)|\min(1,|y|^{-n-\sigma})\ud y.
$
\begin{thm}\label{thm:parabolic evans-krylov2}
 Let $u: \R^n\times[-2,0]\to \R$ be a viscosity solution of \eqref{eq:parabolic} and satisfy \eqref{eq:assuption global lip}. Then there exists $\bar\beta\in (0,1)$ depending only on $n,\lda,\Lda$ such that for $\sigma+\bar\beta-2\ge\gamma_0>0$ we have
\[
\begin{split}
 \|u_t\|_{C^{\bar\al}_{x,t}(B_1\times [-1,0])} & + \|\nabla_x^2 u\|_{C^{\bar\al}_{x,t}(B_{1}\times [-1,0])}\\
& \le C(\|u\|_{L^\infty(\R^n\times[-2,0])}+[u]_{C^{0,1}((t_1,t_2]; L^1(\omega_\sigma))}),
\end{split}
\]
where $\bar\al=\gamma_0\bar\beta/2$ and $C$ is a positive constant depending only on $n,\lda,\Lda$ and $\gamma_0$.
\end{thm}
\begin{proof}
It is the same as the proof of Theorem \ref{thm:parabolic evans-krylov}, except that we use Corollary 7.2 and Corollary 7.4 in \cite{LD3} instead of  Theorem 6.2 in \cite{LD1} and Theorem 4.1 in \cite{LD2},
\end{proof}

\appendix
\section{Appendix: approximation lemmas}

Our proof of Schauder estimates uses perturbative arguments, and we need the following two approximation lemmas, which are variants of Lemma 7 in [6]. We will do a few modifications for our own purposes, and we decide to include them in this appendix for completeness and convenience.

To start with, we recall some definitions and notations about nonlocal elliptic operators, which can be found in \cite{CS09,CS10}. Let $\sigma_0\in (0,2)$ be fixed, and $\omega(y)=(1+|y|^{n+\sigma_0})^{-1}$. We say $u\in L^1(\R^n,\omega)$ if $\int_{\R^n}|u(y)|\omega(y)\ud y<\infty.$ Let $\Omega$ be an open subset of $\R^n$. Let us recall Definition 21 in \cite{CS10} for nonlocal operators. A nonlocal operator $I$ in $\Omega$ is a rule that assigns a function $u$ to a value $I(u,x)$ at every point $x\in\Omega$ satisfying the following assumptions:
\begin{itemize}
 \item $I(u,x)$ is well-defined as long as $u\in C^2(x)$ and $u\in L^1(\R^n,\omega)$;
 \item If $u\in C^2(\Omega)\cap L^1(\R^n,\omega)$, then $I(u,x)$ is continuous in $\Omega$ as a function of $x$.
\end{itemize}
Here $u\in C^2(x)$ we mean that there is a quadratic polynomial $p$ such that $u(y)=p(y)+o(|y-x|^2)$ for $y$ close to $x$. An operator is translation invariant  if $\tau_z Iu=I(\tau_z u)$ where $\tau_z$ is the translation operator $\tau_zu(x)=u(x-z)$.

Given such a nonlocal operator $I$, one can defined a norm $\|I\|$ as in Definition 22 in \cite{CS10}.
We also define a (weaker) norm $\|I\|_{*}$ for our own purpose:
\be\label{eq:norm star}
 \begin{split}
  \|I\|_*:=& \sup\{|I(u,x)|/(1+M):x\in\Omega, u\in C^2(x), \|u\|_{L^\infty(\R^n)}\le M,\\
          &|u(y)-u(x)-(y-x)\cdot \nabla u(x)|\le \frac{M}{2}|x-y|^2 \mbox{ for every }y\in B_1(x)\}.
 \end{split}
\ee
We say that a nonlocal operator $I$ is uniformly elliptic with respect to $\mathscr{L}_0(\lda,\Lda,\sigma)$, which will be written as $\mathscr{L}_0(\sigma)$ for short, if
\[
\mathcal M^-_{\mathscr L_0(\sigma)} v(x)\le I(u+v,x)-I(u,x)\le \mathcal M^+_{\mathscr L_0(\sigma)} v(x),
\]
where
\[
\begin{split}
 \mathcal M^-_{\mathscr L_0(\sigma)} v(x)=\inf_{L\in \mathscr L_0(\sigma)}Lv(x)=(2-\sigma)\int_{\R^n}\frac{\lda\delta v(x,y)^+-\Lda\delta v(x,y)^-}{|y|^{n+\sigma}}\ud y\\
\mathcal M^+_{\mathscr L_0(\sigma)} v(x)=\sup_{L\in \mathscr L_0(\sigma)}Lv(x)=(2-\sigma)\int_{\R^n}\frac{\Lda\delta v(x,y)^+-\lda\delta v(x,y)^-}{|y|^{n+\sigma}}\ud y.\\
\end{split}
\]
It is also convenient to define the limit operators when $\sigma\to 2$ as
\[
 \begin{split}
  \mathcal  M^-_{\mathscr L_0(2)} v(x)=\lim_{\sigma\to 2}\mathcal M^-_{\mathscr L_0(\sigma)} v(x)\\
\mathcal M^+_{\mathscr L_0(2)} v(x)=\lim_{\sigma\to 2}\mathcal M^+_{\mathscr L_0(\sigma)} v(x).
 \end{split}
\]
It has been explained in \cite{CS10} that $\mathcal M^+_{\mathscr L_0(2)}$ is a second order uniformly elliptic operator, whose ellipticity constants $\tilde \lda$ and $\tilde \Lda$ depend only $\lda,\Lda$ and the dimension $n$. Moreover, $\mathcal M^+_{\mathscr L_0(2)} v\le \mathcal M^+(\nabla^2 v)$, where $\mathcal M^+(\nabla^2 v)$ is the second order Pucci operator with ellipticity constants $\tilde \lda$ and $\tilde \Lda$. Similarly, we also have corresponding relations for $ \mathcal M^-_{\mathscr L_0(2)}$.

Our approximation lemmas will be proved by compactness arguments, where we need the concepts of the weak convergence of nonlocal operators in Definition 41 in \cite{CS10}. We say that a sequence of nonlocal operators $I_k\rightharpoonup I$ weakly in $\Omega$ if, for every $x_0\in\Omega$ and for every function $v$ of the form
\[
 v(x)=\begin{cases}
       p(x)\quad\mbox{if }|x-x_0|\le r;\\
       u(x)\quad\mbox{if }|x-x_0|> r,
      \end{cases}
\]
where $p$ is a polynomial of degree two and $u\in L^1(\R^n,\omega)$, we have $I_k(v,x)\to I(v,x)$ uniformly in $B_{r/2}(x_0)$.
\begin{lem}\label{lem:appr0}
For some $\sigma\ge \sigma_0>0$ we consider nonlocal operators $I_0$, $I_1$ and $I_2$ uniformly elliptic with respect to $\mathscr{L}_0(\sigma)$. Assume also that $I_0$ is translation invariant and $I_0 (0)=1$.

 Given $M>0$, a modulus of continuity $\omega_1$ and $\va>0$, there exists $\eta_1$ (small, independent of $\sigma$) and $R$ (large, independent of $\sigma$) so that if $u,v,I_0,I_1$ and $I_2$ satisfy
\[
\begin{split}
I_0 (v,x)=0, \quad 
I_1 (u,x)\ge -\eta_1, \quad 
I_2 (u,x)\le \eta_1 \quad \mbox{in }B_1
\end{split}
\]
in viscosity sense, and
\[
\begin{split}
\|I_1-I_0\|_*\le\eta_1,\quad \|I_2-I_0\|_*\le\eta_1\quad \mbox{in }B_1,\\
\end{split}
\]
and
\[
\begin{split}
u&=v  \quad \mbox{in }\R^n\setminus B_1,\\
|u(x)|&\le M \quad \mbox{in }\R^n,\\
|u(x)-u(y)|&\le \omega_1(|x-y|) \quad \mbox{for every }x\in B_R\setminus B_1\mbox{ and }y\in \R^n\setminus B_{1},\\
\end{split}
\]
then $|u-v|\le\va$ in $B_1$.
\end{lem}
\begin{proof}
 It follows from the proof of Lemma 7\footnote{The statements of Lemma 7 and Lemma 8 in \cite{CS10} should be read under the condition that $I_0$ is translation invariant (see \cite{LS}), which does not affect their applications in \cite{CS10}.} in \cite{CS10} with modifications. We argue by contradiction. Suppose the above lemma was false. Then there would be sequences $\sigma_k$, $I_0^{(k)}$, $I_1^{(k)}$, $I_2^{(k)}$, $\eta_k$, $u_k$, $v_k$ such that $\sigma_k\to \sigma\in [\sigma_0,2]$, $\eta_k\to 0$ and all the assumptions of the lemma are valid, but $\sup _{B_1}|u_k-v_k|\ge\va$.

Since $I_0^{(k)}$ is a sequence of uniformly elliptic translation invariant operators with respect to $\mathscr L(\sigma_k)$, by Theorem 42 in \cite{CS10} that we can take a subsequence, which is still denoted as $I_0^{(k)}$, that converges weakly to some nonlocal operator $I_0$, and $I_0$ is also translation invariant, and  uniformly elliptic with respect to the class $\mathscr{L}_0(\sigma)$.

It follows from the boundary regularity Theorem 32 in \cite{CS10} that $u_k$ and $v_k$ have a modulus of continuity, uniform in $k$, in the closed unit ball $\overline B_1.$ Thus, $u_k$ and $v_k$ have a uniform (in $k$) modulus of continuity on $B_{R_k}$ with $R_k\to\infty$. We can subsequences of $\{u_k\}$ and $\{v_k\}$, which will be still denoted as $\{u_k\}$ and $\{v_k\}$, which converges locally uniformly in $\R^n$ to $u$ and $v$, respectively. Moreover, $u=v$ in $\R^n\setminus B_1$, and $\sup_{B_1}|u-v|\ge\va$.

In the following, we are going to show that
\be\label{eq:limit same}
 I_0(u,x)=0=I_0(v,x) \quad\mbox{in } B_1,
\ee
from which we can conclude that $u\equiv v$ in $B_1$, since $I_0$ is translation invariant. But we know that $\sup_{B_1}|u-v|\ge\va$. This reaches a contradiction.

The second equality of \eqref{eq:limit same} follows from Lemma 5 in \cite{CS10}. The first equality actually follows almost identically from the proof of Lemma 5 in \cite{CS10}: we only need to notice that the sequence $\{u_k\}$ is uniformly bounded by $M$, and thus the conditions that $I_1^{(k)}(u_k,x)\ge-\eta_k$, $I_2^{(k)}(u_k,x)\le\eta_k$, $\|I_1^{(k)}-I_0^{(k)}\|_*\to 0$ and $\|I_2^{(k)}-I_0^{(k)}\|_*\to 0$ are sufficient to show $I_0(u,x)=0$ in $B_1$ as in the proof of Lemma 5 in \cite{CS10}.
\end{proof}

\begin{lem}\label{lem:appr}
For some $\sigma\ge \sigma_0>0$ we consider nonlocal operators $I_0$, $I_1$ and $I_2$ uniformly elliptic with respect to $\mathscr{L}_0(\sigma)$. Assume also that
\[
I_0 v(x):=\inf_{a\in\mathcal A}\left\{ \int_{\R^n} \delta v(x,y)K_a(y)\ud y+h_a(x)\right\} \quad\mbox{in }B_4,
\]
where each $K_a\in\mathscr L_2(\sigma)$ and for some constant $\beta\in (0,1)$,
\[
[h_a]_{C^\beta(B_4)} \le M_0\quad\mbox{and}\quad\inf_{a\in\mathcal A} h_a(x)=0~\forall~x\in B_4.
\]
Given $M_0,M_1, M_2, M_3>0$, $R_0>5$, $\beta, \nu\in (0,1)$, and $\va>0$, there exists $\eta_2$ (small, independent of $\sigma$) so that if $u,v,I_0,I_1$ and $I_2$ satisfy
\[
\begin{split}
I_0 (v,x)=0, \quad 
I_1 (u,x)\ge -\eta_2, \quad 
I_2 (u,x)\le \eta_2 \quad \mbox{in }B_4,\\
\end{split}
\]
in viscosity sense, and
\[
\|I_1-I_0\|_*\le\eta_2,\quad
\|I_2-I_0\|_*\le\eta_2\quad \mbox{in }B_4,\\
\]
and
\[
\begin{split}
u&=v  \quad \mbox{in }\R^n\setminus B_4,\\
u&\equiv 0 \quad \mbox{in }\R^n\setminus B_{R_0},\\
|u|&\le M_1 \quad \mbox{in }\R^n,\\
[u]_{C^\nu(B_{R_0-\tau})}&\le M_2\tau^{-4}\quad\forall~\tau\in (0,1),\\
\|v\|_{C^{\sigma+\beta}(B_{4-\tau})}&\le M_3\tau^{-4}\quad\forall~\tau\in (0,1),\\
\end{split}
\]
then $|u-v|\le\va$ in $B_4$.
\end{lem}
\begin{proof}
This lemma can be proved similarly to Lemma \ref{lem:appr0}. Suppose the above lemma was false. Then there would be sequences $\sigma_k$, $I_0^{(k)}$, $I_1^{(k)}$, $I_2^{(k)}$, $\eta_k$, $u_k$, $v_k$ such that $\sigma_k\to\sigma\in [\sigma_0,2]$, $\eta_k\to 0$ and all the assumptions of the lemma are valid, but $\sup _{B_1}|u_k-v_k|\ge\va$.

By our assumptions, it is clear that, up to a subsequence, $u_k$ converges locally uniformly in $B_{R_0}$. Since $u_k\equiv 0$ in $\R^n\setminus B_{R_0}$, it converges almost everywhere to some function $u$ in $\R^n$. Since $v_k$ is bounded and has a modulus continuity on $B_5\setminus B_4$, then by the boundary regularity Theorem 32 in \cite{CS10}, there is another modulus continuity that extends to the closed unit ball $\overline B_4$, and thus, $v_k$ converges uniformly in $\overline B_4$, as well as in $C^{\sigma+\beta-\mu}_{loc}(B_4)$ for any arbitrarily small $\mu>0$. Therefore, $v_k$ converges to some function $v\in C^{\sigma+\beta-\mu}_{loc}(B_4)$ almost everywhere in $\R^n$. Moreover, $u=v$ in $\R^n\setminus B_4$, and $\sup_{B_4}|u-v|\ge\va$.

We are going to show that there exists a subsequence of $\{I_0^{(k)}\}$, which is still denoted as $I_0^{(k)}$, that converges weakly in $B_4$ to some nonlocal operator $I_0$, and $I_0$ is uniformly elliptic with respect to the class $\mathscr{L}_0(\sigma)$. Then it follows from the proof of \eqref{eq:limit same} that $u$ and $v$ solve the same equation $I_0(u,x)=I_0(v,x)=0$ in $B_4$ in viscosity sense. Since $v\in C^{\sigma+\beta-\mu}_{loc}(B_4)$ is a classical solution and $u=v$ in $\R^n\setminus B_4$, we have $u=v$ in $B_4$, which is a contradiction.

The proof of that there exists a subsequence of $\{I_0^{(k)}\}$ weakly converges in $B_4$ will basically follow from the proofs of Lemma 6 and Theorem 42 in \cite{CS10}.

\emph{Claim 1:} Let $\varphi$ be a function
\[
\varphi(x)=\begin{cases}
p(x)\quad\mbox{in }B_r\\
\Phi(x)\quad\mbox{in }\R^n\setminus B_r,
\end{cases}
\]
where $r>0$, $p(x)$ is a second order polynomial, and $\Phi\in L^1(\R^n,\omega)$. Then there exists a subsequence $\{I_0^{(k_j)}\}$ such that $f_{k_j}(x):=I_0^{(k_j)}\varphi(x)$ converges uniformly in $B_{r/2}$.

\emph{Proof of Claim 1:} Since $I_0^{(k)}(0)=0$, by uniformly ellipticity, $f_{k}$ is uniformly bounded in $\overline B_{r/2}$. We are going to find a uniform modulus of continuity for $f_k$ in $\overline B_{r/2}$ so that Claim 1 follows from Arzela-Ascoli theorem. 

Recall $\tau_z\varphi(x)=\varphi(x+z)$. Given $x,y\in B_{r/2}$ with $|x-y|<r/8$, we have
\[
f_k(x)-f_k(y)\le \mathcal M^+_{\mathscr L(\sigma_k)}(v-\tau_{y-x}v, x)+M_0|x-y|^\beta,
\]
where the first term has a modulus of continuity depends on $\varphi$ but not $I_0^{(k)}$ as shown in the proof of Lemma 6 in \cite{CS10}. This finishes the proof of Claim 1.

As long as we have Claim 1, it follows from the proof of Theorem 42 identically that there exists a subsequence of $\{I_0^{(k)}\}$, which is still denoted as $I_0^{(k)}$, that converges weakly in $B_4$ to some nonlocal operator $I_0$, and $I_0$ is uniformly elliptic with respect to the class $\mathscr{L}_0(\sigma)$.
\end{proof}

\bigskip

\noindent Tianling Jin

\noindent Department of Mathematics, The University of Chicago\\ 
5734 S. University Ave, Chicago, IL 60637, USA\\[1mm]
Email: \textsf{tj@math.uchicago.edu}

\medskip

\noindent Jingang Xiong

\noindent Beijing International Center for Mathematical Research\\ 
Peking University, Beijing 100871, China\\[1mm]
Email: \textsf{jxiong@math.pku.edu.cn}

\end{document}